\documentclass[10pt]{article}
\usepackage{lmodern}
\usepackage{amsmath}
\usepackage[T1]{fontenc}
\usepackage[utf8]{inputenc}
\usepackage{authblk}
\usepackage{amsfonts}
\usepackage{graphicx}
\usepackage{rotating}
\usepackage{amssymb}
\usepackage[english]{babel}
\usepackage{color}
\usepackage{amsthm}
\usepackage{graphicx}
\usepackage{mathrsfs}
\usepackage{makecell}
\usepackage{microtype}
\usepackage{enumerate}
\usepackage{mathscinet}
\usepackage{array}
\usepackage{multirow}
\usepackage{booktabs}
\usepackage{enumerate}
\usepackage{booktabs}
\usepackage[cal=boondoxo,bb=ams]{mathalfa}
\usepackage{hyperref}
\hypersetup{hidelinks}
\usepackage{titlesec}
\newtheorem{theorem}{Theorem}[section]
\newtheorem{prop}{Proposition}[section]
\newtheorem{lemma}{Lemma}[section]

\newtheorem{remark}{Remark}[section]

\newcommand{\ml}{\mathcal}
\newcommand{\mb}{\mathbb}

\DeclareMathOperator{\intt}{int}
\DeclareMathOperator{\extt}{ext}
\DeclareMathOperator{\bdd}{bdd}

\title{Asymptotic profiles and singular limits for the viscoelastic damped wave equation with memory of type I}
\author[1]{Wenhui Chen\thanks{Corresponding author: Wenhui Chen (wenhui.chen.math@gmail.com)}}
\affil[1]{School of Mathematical Sciences, Shanghai Jiao Tong University, 200240 Shanghai, China}

\date{}
\begin{document}

\maketitle
\begin{abstract}
In this paper, we are interested in the Cauchy problem for the viscoelastic damped wave equation with memory of type I. By applying WKB analysis and Fourier analysis, we explain the memory's influence on dissipative structures and asymptotic profiles of solutions to the model with weighted $L^1$ initial data. Furthermore, concerning standard energy and the solution itself, we establish singular limit relations between the Moore-Gibson-Thompson equation with memory and the viscoelastic damped wave equation with memory.\\
	
	
	\noindent\textbf{AMS Classification (2020)} Primary: 35B40, 35L15; Secondary: 74D05, 35L05, 34K26, 35C20.
\end{abstract}
\fontsize{12}{15}
\selectfont

\section{Introduction}\label{Sec_Intro}
In this paper, we consider the following Cauchy problem for the viscoelastic damped wave equation with memory:
\begin{align}\label{VDW_Equation}
\begin{cases}
u_{tt}-\Delta u-\Delta u_t+g\ast \Delta u=0,&x\in\mb{R}^n,\ t>0,\\
(u,u_t)(0,x)=(u_0,u_1)(x),&x\in\mb{R}^n,
\end{cases}
\end{align}
where $g=g(t)$ stands for the time-dependent relaxation function with an exponential decay property such that
\begin{align}\label{Exponential_Decay_Relax}
g(t):=\mathrm{e}^{-\gamma t}\ \ \mbox{with}\ \ \gamma>1,
\end{align}
which is the one of the most relevant case in the connection with evolution equations with memory, e.g. \cite{DellOroPata2017-Milan,Mori-Timoshenko,DellOroLasieckaPata2019,Liu-Ueda-2020,Chen-Dao2020}. Here, the convolution term with respect to the time variable is denoted by
\begin{align}\label{Memory_Type_I}
(g\ast\Delta u)(t,x):=\int_0^tg(t-\eta)\Delta u(\eta,x)\mathrm{d}\eta.
\end{align}
According to the classification of memory in \cite{LasieckaWang2016}, the memory shown in \eqref{Memory_Type_I} is the so-called type I, in which the integral itself suggests the non-locality in time and it somehow leads that the equation remembers information in the past history. Concerning the boundary value problem, the memory effect in semilinear viscoelastic damped wave equations has been investigated by \cite{Borini-Pata-1999,Plinio-Pata,Plinio-Pata-Zelik} from the viewpoint of the unique existence of global attractors. However, so far to the best of the authors' knowledge, the Cauchy problem for the viscoelastic damped wave equation with memory does not been studied. We will answer in this paper for some qualitative properties of solutions.

Let us present a historical overview on some results for the linear evolution equations in $\mb{R}^n$, which are strongly linked with our model and the motivation for us to consider the viscoelastic damped wave equation with memory.

In recent years, the Cauchy problem for the viscoelastic damped wave equation, namely,
\begin{align}\label{Eq_Visco_Dam_Wav_NOT_Memory}
\begin{cases}
u_{tt}-\Delta u-\Delta u_t=0,&x\in\mb{R}^n,\ t>0,\\
(u,u_t)(0,x)=(u_0,u_1)(x),&x\in\mb{R}^n.
\end{cases}
\end{align}
have caught a lot of attention, where the pioneering work was introduced by \cite{Ponce1985} more than thirty years ago. Concerning dissipative structures of the Cauchy problem \eqref{Eq_Visco_Dam_Wav_NOT_Memory}, the author of \cite{Shibata2000} investigated some $L^p-L^q$ estimates not necessary on the conjecture line. Later, some $L^2-L^2$ estimates with additional $L^1$ regularity were derived in \cite{IkehataNatsume,Charao-Da-Ike-2013,D'A-Re-2014}. For another, carrying suitable weighted $L^1$ data, a series paper \cite{Ike-Tod-Yor-2013,Ike-2014,Ike-Ono-2017,BarreraVolkmer2019,BarreraVolkmer2020} derived the diffusion waves type asymptotic profiles of solutions and optimal estimates in the $L^2$ norm. Let us turn to our model \eqref{VDW_Equation}. We found that by formally taking $g(t)\equiv0$, the Cauchy problem \eqref{VDW_Equation} will immediately become the viscoelastic damped wave equation \eqref{Eq_Visco_Dam_Wav_NOT_Memory}. However, as describing in \cite{LasieckaWang2016}, the more is not necessarily the better. For example, \cite{Fabrizio-Poli-2002} observed that the additional memory would destroy exponential stabilities of a damped wave equation. For this reason, we are interested in the influence of memory on some qualitative properties of solutions to the viscoelastic damped wave equation, especially, dissipative structures and asymptotic profiles. But the study of our model \eqref{VDW_Equation} is not a simple generalization of the study of the memoryless model \eqref{Eq_Visco_Dam_Wav_NOT_Memory} because the additional memory term will arise some difficulties when we employed Fourier analysis and energy method. 

Next, we turn to a related model to the Cauchy problem \eqref{VDW_Equation}. Let us introduce the Cauchy problem for the Moore-Gibson-Thompson (MGT) equation in the dissipative case as follows:
\begin{align}\label{Eq_MGT_Memory}
\begin{cases}
\tau u_{ttt}+u_{tt}-\Delta u-\Delta u_t+g\ast \Delta u=0,&x\in\mb{R}^n,\ t>0,\\
(u,u_t,u_{tt})(0,x)=(u_0,u_1,u_2)(x),&x\in\mb{R}^n,
\end{cases}
\end{align}
where the thermal relaxation $\tau\in(0,1)$ in the view of the physical context of acoustic waves. The MGT equation mainly describes the wave propagation in viscous thermally relaxing fluids. Concerning the memoryless case, i.e. $g(t)\equiv0$, the linear MGT equation has been widely studied in \cite{MooreGibson1960,Thompson1972,Kala-Tiry-1997,Gorain2010,KaltenbacherLasieckaMarchand2011,KaltenbacherLasiecka2012,MarchandMcDevittTriggiani2012,DellOroPata2017,PellicerSaiHouari2017,ChenPalmieri201901,BucciEller2020,C.-Ike-2020} and references therein. Quite recently, some memory effects in the MGT equation were found in \cite{LasieckaWang2016,LasieckaWang2015,DellOroLasieckaPata2016,Lasiecka2017,AlvesCaixetaSilvaRodrigues2018,DellOroLasieckaPata2019,BucciPandolfi,Nikoic-Said2020,Bounadja-SaidHouari2020}, where the relaxation function $g(t)$ decays exponentially or was supposed with some general decay properties. If we formally set $\tau=0$, the MGT equation with memory \eqref{Eq_MGT_Memory} will automatically shift into our model \eqref{VDW_Equation}. Therefore, it seems interesting to study the influence from the thermal relaxation $\tau$, particularly, the limit relation of solutions between our model \eqref{VDW_Equation} and the MGT equation \eqref{Eq_MGT_Memory}.  We should emphasize that there are a lot of differences between the cases $\tau>0$ and $\tau=0$, for instance, the MGT equation is a strictly hyperbolic equation (well-posedness for the Cauchy problem is well-studied) while the viscoelastic damped wave is classified into a 2-evolution equation.

Our first aim in the paper is to understand some qualitative properties of solutions to the viscoelastic damped wave equation \eqref{VDW_Equation} with the exponential decay memory \eqref{Exponential_Decay_Relax}. 
 Then, we will explain dissipative structures of the model by the mean of asymptotic expansions of eigenvalues in Section \ref{Sec_Asy_Beh_Solution}. Later, thanks to Fourier analysis, asymptotic profiles of solutions in a framework of weighted $L^1$ space will be deduced in Section \ref{Sec_Exp_Asy_Prof}, where optimal estimates of the solution in the $L^2$ norm will be derived. Furthermore, the profiles of the solution itself can be understood as a combination of the diffusion waves and the higher-order diffusion waves.

Regarding the limit case as taking $\tau=0$ in the Cauchy problem \eqref{Eq_MGT_Memory}, we can formally transfer the MGT equation with memory \eqref{Eq_MGT_Memory} to the viscoelastic damped wave with memory \eqref{VDW_Equation}. To study the convergence between their solutions as $\tau\downarrow0$, we will investigate the singular limit problem in Section \ref{Sec_Sing_Lim}. We focus on the influence from different assumptions on initial data to the convergence rate with respect to $\tau$, where the singular limit relations for standard energy and the solution itself will be established, respectively. Concerning the singular limit for the solution itself, we will construct some new energies with respect to the integral form of the solution.

\medskip

\noindent\textbf{Notation: } We give some notations to be used in this paper. Later, $c$ and $C$ denote some positive constants, which may be changed from
line to line. We denote that $f\lesssim g$ if there exists a positive constant $C$ such that $f\leqslant Cg$ and, analogously, for $f\gtrsim g$. Moreover, $\ml{I}$ denotes the identity operator such that $\ml{I}:f\to\ml{I}f=f$. Next, $|D|^s$ with $s\geqslant0$ denotes the pseudo-differential operator with the symbol $|\xi|^s$. Lastly, let us define the some zones for the Fourier space as follows:
\begin{align*}
\ml{Z}_{\intt}(\varepsilon)&:=\{\xi\in\mb{R}^n:|\xi|<\varepsilon\ll1\},\\
\ml{Z}_{\bdd}(\varepsilon,N)&:=\{\xi\in\mb{R}^n:\varepsilon\leqslant |\xi|\leqslant N\},\\
\ml{Z}_{\extt}(N)&:=\{\xi\in\mb{R}^n:|\xi|> N\gg1\}.
\end{align*}
The cut-off functions $\chi_{\intt}(\xi),\chi_{\text{bdd}}(\xi),\chi_{\extt}(\xi)\in \mathcal{C}^{\infty}(\mb{R}^n)$ having their supports in the zone $\ml{Z}_{\intt}(\varepsilon)$, $\ml{Z}_{\text{bdd}}(\varepsilon/2,2N)$ and $\ml{Z}_{\extt}(N)$, respectively, so that $\chi_{\intt}(\xi)+\chi_{\bdd}(\xi)+\chi_{\extt}(\xi)=1$.

\section{Asymptotic behavior of solutions}\label{Sec_Asy_Beh_Solution}
As a matter of fact, we may transfer the integro-differential equation \eqref{VDW_Equation} to the first-order (in time) differential coupled system. Although one may apply spectral theory associated with asymptotic analysis to derive asymptotic solution formula to that coupled system, it is still a challenging work to analyze precisely some qualitative properties of the solution itself. In order to overcome this difficulty, we will understand the Cauchy problem \eqref{VDW_Equation} according to a higher-order (in time) equation.

With the aim of treating the Cauchy problem \eqref{VDW_Equation}, we firstly apply the partial Fourier transform with respect to spatial variables such that $\hat{u}(t,\xi)=\ml{F}_{x\to\xi}(u(t,x))$. Thus, the next initial value problem for $|\xi|$-dependent integro-differential equation appears:
\begin{align}\label{VDW_Fourier}
\begin{cases}
\hat{u}_{tt}+|\xi|^2\hat{u}+|\xi|^2\hat{u}_t-|\xi|^2g\ast\hat{u}=0,&\xi\in\mb{R}^n,\ t>0,\\
(\hat{u},\hat{u}_t)(0,\xi)=(\hat{u}_0,\hat{u}_1)(\xi),&\xi\in\mb{R}^n.
\end{cases}
\end{align}
The relaxation function \eqref{Exponential_Decay_Relax} shows that $g(0)=1$, $g'(t)=-\gamma g(t)$ and, consequently, 
\begin{align}\label{Rela_01}
\frac{\mathrm{d}}{\mathrm{d}t}\left(-|\xi|^2g\ast\hat{u}\right)=-|\xi|^2\hat{u}+\gamma|\xi|^2g\ast\hat{u}.
\end{align}
We found that the solution $\hat{u}=\hat{u}(t,\xi)$ to the initial value problem
\begin{align}\label{Eq_01}
\begin{cases}
\displaystyle{\left(\frac{\mathrm{d}}{\mathrm{d}t}+\gamma\ml{I}\right)\left(\hat{u}_{tt}+|\xi|^2\hat{u}+|\xi|^2\hat{u}_t-|\xi|^2g\ast\hat{u}\right)=0,}&\xi\in\mb{R}^n,\ t>0,\\
(\hat{u},\hat{u}_t,\hat{u}_{tt})(0,\xi)=(\hat{u}_0,\hat{u}_1,\hat{u}_2)(\xi),&\xi\in\mb{R}^n,
\end{cases}
\end{align}
equipping $\hat{u}_2(\xi):=-|\xi|^2(\hat{u}_0(\xi)+\hat{u}_1(\xi))$, is exactly the solution to the initial value problem \eqref{VDW_Fourier}. This statement can be proved by multiplying the equation in \eqref{Eq_01} by $\mathrm{e}^{\gamma t}$ and integrating the resultant equation over $[0,t]$. 
In other words, with the help of \eqref{Rela_01}, we will study the qualitative properties of solutions to \eqref{VDW_Equation} by considering the next third-order initial value problem:
\begin{align}\label{Eq_02}
\begin{cases}
\hat{u}_{ttt}+\left(|\xi|^2+\gamma\right)\hat{u}_{tt}+(1+\gamma)|\xi|^2\hat{u}_t+(\gamma-1)|\xi|^2\hat{u}=0,&\xi\in\mb{R}^n,\ t>0,\\
(\hat{u},\hat{u}_t,\hat{u}_{tt})(0,\xi)=(\hat{u}_0,\hat{u}_1,\hat{u}_2)(\xi),&\xi\in\mb{R}^n.
\end{cases}
\end{align}
\subsection{Asymptotic expansions of eigenvalues}
One may found that the characteristic equation to the equation in the initial value problem \eqref{Eq_02} is given by
\begin{align}\label{Eq_Eigen}
\lambda^3+(|\xi|^2+\gamma)\lambda^2+(1+\gamma)|\xi|^2\lambda+(\gamma-1)|\xi|^2=0,
\end{align}
whose roots can be written by $\lambda_j=\lambda_j(|\xi|)$ with $j=1,2,3$. Although Cardano's formula is effective to solve a cubic equation explicitly, it seems a challenging work to analyze the behavior of solutions by using these roots of complex forms. To overcome this difficulty, we will apply asymptotic expansions of eigenvalues for $\xi\in\ml{Z}_{\intt}(\varepsilon)$ as well as $\xi\in\ml{Z}_{\extt}(N)$, respectively.

Let us explain nature of the roots of the cubic equation \eqref{Eq_Eigen}. The corresponding discriminant $\triangle_{\mathrm{Cub}}$ is given by
\begin{align*}
\triangle_{\mathrm{Cub}}&=|\xi|^2\left((\gamma^2-2\gamma+5)|\xi|^6-2(\gamma^3+\gamma^2-\gamma+11)|\xi|^4\right.\\
&\quad\left.+(\gamma^4+8\gamma^3-14\gamma^2+36\gamma-27)|\xi|^2-4\gamma^3(\gamma-1)\right).
\end{align*}
For one thing, concerning $\xi\in\ml{Z}_{\intt}(\varepsilon)$ with $\varepsilon\ll 1$, due to $-\gamma^3(\gamma-1)<0$ for all $\gamma>1$, we conclude $\triangle_{\mathrm{Cub}}<0$. For another, concerning $\xi\in\ml{Z}_{\extt}(N)$ with $N\gg1$, due to $\gamma^2-2\gamma+5>0$ for all $\gamma>1$, we claim $\triangle_{\mathrm{Cub}}>0$. It means that for small frequencies, the cubic equation \eqref{Eq_Eigen} has one real root and two non-real complex conjecture roots, while for large frequencies, the cubic equation \eqref{Eq_Eigen} has three distinct real roots. Finally, the case for multiple roots, i.e. $\triangle_{\mathrm{Cub}}=0$, appears when $\xi\in\ml{Z}_{\mathrm{bdd}}(\varepsilon,N)$ only.\\

\noindent\textbf{Part 1. Asymptotic expansions for small frequencies}\\
 We derive the eigenvalues $\lambda_j(|\xi|)$ with $j=1,2,3$ owing the asymptotic expansions for $\xi\in\ml{Z}_{\intt}(\varepsilon)$ such that
\begin{align}\label{Expansion_small}
\lambda_j(|\xi|)=\lambda_j^{(0)}+\lambda_j^{(1)}|\xi|+\lambda_j^{(2)}|\xi|^2+\lambda_j^{(3)}|\xi|^3+\cdots,
\end{align}
where the coefficients $\lambda_j^{(k)}\in\mb{C}$ for all $k\in\mb{N}_0$. Let us plug \eqref{Expansion_small} into \eqref{Eq_Eigen} and process lengthy but straightforward calculations, until different characteristic roots appear. Consequently, we have
\begin{align*}
\lambda_{1,2}(|\xi|)&=\pm i\sqrt{\frac{\gamma-1}{\gamma}}|\xi|-\frac{\gamma^2+1}{2\gamma^2}|\xi|^2+\ml{O}(|\xi|^3),\\
\lambda_3(|\xi|)&=-\gamma+\frac{1}{\gamma^2}|\xi|^2+\ml{O}(|\xi|^3),
\end{align*}
with $\xi\in\ml{Z}_{\intt}(\varepsilon)$. Here, $\gamma>1$ was used. Particularly, formally taking $\gamma\to\infty$, the principal parts of $\lambda_{1}(|\xi|)$ and $\lambda_2(|\xi|)$ correspond to those for the viscoelastic damped wave equation without memory (see \cite{D'A-Re-2014,Ike-2014}).\\

\noindent\textbf{Part 2. Asymptotic expansions for large frequencies}\\
In the case $\xi\in\ml{Z}_{\extt}(N)$, the eigenvalues have the asymptotic expansions such that
\begin{align}\label{Expansion_large}
\lambda_j(|\xi|)=\bar{\lambda}_j^{(0)}|\xi|^2+\bar{\lambda}_j^{(1)}|\xi|+\bar{\lambda}_j^{(0)}+\bar{\lambda}_j^{(3)}|\xi|^{-1}+\cdots,
\end{align}
where the coefficients $\bar{\lambda}_j^{(k)}\in\mb{C}$ for all $k\in\mb{N}_0$. By substituting \eqref{Expansion_large} into \eqref{Eq_Eigen}, we may arrive at
\begin{align*}
\lambda_{1,2}(|\xi|)&=-\frac{1+\gamma\pm\sqrt{(1+\gamma)^2-4(\gamma-1)}}{2}+\ml{O}(|\xi|^{-1}),\\
\lambda_3(|\xi|)&=-|\xi|^2+1+\ml{O}(|\xi|^{-1}),
\end{align*}
with $\xi\in\ml{Z}_{\extt}(N)$. Particularly, formally taking $\gamma\to\infty$, the principal parts of eigenvalues $\lambda_{2}(|\xi|)$ and $\lambda_3(|\xi|)$ for $\xi\in\ml{Z}_{\extt}(N)$ correspond to those for the viscoelastic damped wave equation without memory (see \cite{D'A-Re-2014,Ike-2014}), where we applied
\begin{align*}
\lim\limits_{\gamma\to\infty}\frac{1+\gamma-\sqrt{(1+\gamma)^2-4(\gamma-1)}}{2}=\lim\limits_{\gamma\to\infty}\frac{2(\gamma-1)}{1+\gamma+\sqrt{(1+\gamma)^2-4(\gamma-1)}}=1.
\end{align*}

\noindent\textbf{Part 3. Stability of eigenvalues for bounded frequencies}\\
Considering frequencies localizing in the bounded zone $\ml{Z}_{\bdd}(\varepsilon,N)$, we claim that $\mathrm{Re}\,\lambda_j(|\xi|)<0$ for all $j=1,2,3$. Let us prove this assertion by a contradiction argument. Let us assume there is a nontrivial root $\lambda=id$ to the characteristic equation \eqref{Eq_Eigen}. Namely,
\begin{align*}
-(|\xi|^2+\gamma)d^2+(\gamma-1)|\xi|^2=0\ \ \mbox{and}\ \ -id(d^2-(1+\gamma)|\xi|^2)=0.
\end{align*}
Due to $d\neq0$, they yield a contradiction immediately. Hence, there does not exist any pure imaginary roots to the cubic equation \eqref{Eq_Eigen}. According to the compactness of the bounded zone $\ml{Z}_{\bdd}(\varepsilon,N)$ and the continuity of the eigenvalues $\lambda_j(|\xi|)$, we are able to assert that
\begin{align}\label{BDD}
\mathrm{Re}\,\lambda_j(|\xi|)<0\ \ \mbox{for}\ \ j=1,2,3
\end{align}
due to $\mathrm{Re}\,\lambda_j(|\xi|)<0$ for $j=1,2,3$ and $\xi\in\ml{Z}_{\intt}(\varepsilon)\cup\ml{Z}_{\extt}(N)$.

\subsection{Pointwise estimates in the Fourier space}
In the last subsection, we have obtained some asymptotic behaviors and stabilities of eigenvalues in different zones, which bring some useful tools for us to derive some representations of solutions. According to Part 1 and Part 2 in the above discussions, we may observe that the pairwise distinct eigenvalues appear for $\xi\in\ml{Z}_{\intt}(\varepsilon)\cup\ml{Z}_{\extt}(N)$. For this reason, by making use of the representation for initial data $\hat{u}_2(\xi)$, the solution to \eqref{Eq_02} can be represented by
\begin{align}\label{Represent_Solution}
\hat{u}(t,\xi)=\widehat{K}_0(t,|\xi|)\hat{u}_0(\xi)+\widehat{K}_1(t,|\xi|)\hat{u}_1(\xi),
\end{align}
where the kernels are
\begin{align*}
\widehat{K}_0(t,|\xi|)&:=\sum\limits_{j=1,2,3}\frac{\prod\nolimits_{k=1,2,3,\ k\neq j}\lambda_k(|\xi|)-|\xi|^2}{\prod\nolimits_{k=1,2,3,\ k\neq j}(\lambda_j(|\xi|)-\lambda_k(|\xi|))}\mathrm{e}^{\lambda_j(|\xi|)t},\\
\widehat{K}_1(t,|\xi|)&:=\sum\limits_{j=1,2,3}\frac{-\sum\nolimits_{k=1,2,3,\ k\neq j}\lambda_k(|\xi|)-|\xi|^2}{\prod\nolimits_{k=1,2,3,\ k\neq j}(\lambda_j(|\xi|)-\lambda_k(|\xi|))}\mathrm{e}^{\lambda_j(|\xi|)t}.
\end{align*}
\begin{remark}
If one is interested in a general assumption for the relaxation function $g(t)$ such that
\begin{align}\label{General_g(t)}
g(t)>0\ \ \mbox{and} \ \ g'(t)\leqslant-\tilde{c}g(t)\ \ \mbox{for any}\ \ t\geqslant0,
\end{align}
we cannot derive the representation of solutions by following the same approach as the above discussion. At this time, one may apply the Laplace transform with respect to the time variable $\ml{L}_{t\to m}(\hat{u})(m,\xi)=\ml{L}_{t\to m}(\hat{u}(t,\xi))$. By this way, we get
\begin{align*}
\widehat{K}_0(t,|\xi|)&=\ml{L}^{-1}_{m\to t}\left(\frac{m+|\xi|^2}{(m+1)|\xi|^2+m^2-|\xi|^2\ml{L}(g)(m)}\right),\\
\widehat{K}_1(t,|\xi|)&=\ml{L}^{-1}_{m\to t}\left(\frac{1}{(m+1)|\xi|^2+m^2-|\xi|^2\ml{L}(g)(m)}\right).
\end{align*}
Setting $\ml{F}_0(m;|\xi|):=(m+1)|\xi|^2+m^2-|\xi|^2\ml{L}(g)(m)$ with $m\in\mb{C}$, we find that $(m+|\xi|^2)/\ml{F}_0(m;|\xi|)$ and $1/\ml{F}_0(m;|\xi|)$ are analytic in the range
\begin{align*}
\mb{D}:=\left\{m\in\mb{C}:\ \mathrm{Re}\,m>\max\left\{\frac{-(\tilde{c}+1)+\sqrt{(\tilde{c}-1)^2+4g(0)}}{2},0 \right\} \right\}.
\end{align*}
Moreover, due to the assumption \eqref{General_g(t)}, if $\mathrm{Re}\, m>-\tilde{c}$, there exists $\ml{L}(g)(m)$, where we used
\begin{align*}
|\ml{L}(g)(m)|\leqslant\frac{g(0)}{\mathrm{Re}\,m+\tilde{c}}.
\end{align*}
Finally, by following the similar procedure to those in Section 2.1 in \cite{Liu-Ueda-2020}, the existences of $\widehat{K}_0(t,|\xi|)$ and $\widehat{K}_1(t,|\xi|)$ can be proved.
\end{remark}

It is well-known that dissipative properties can be characterized by pointwise estimates in the Fourier space. Let us introduce some useful pointwise estimates on $|\hat{u}(t,\xi)|$ and $|\hat{u}_t(t,\xi)|$, respectively.
\begin{prop}\label{Prop_Pointwise_Est}
	The solution $\hat{u}=\hat{u}(t,\xi)$ to the initial value problem \eqref{Eq_02} with the relaxation function \eqref{Exponential_Decay_Relax} fulfills the following estimates:
	\begin{align}\label{Point_wise_small}
	\chi_{\intt}(\xi)|\hat{u}(t,\xi)|&\lesssim\chi_{\intt}(\xi)\left(\left(|\cos(\tilde{\gamma}|\xi|t)|+|\xi|\,|\sin(\tilde{\gamma}|\xi|t)|\right)\mathrm{e}^{-\frac{\gamma^2+1}{2\gamma^2}|\xi|^2t}+|\xi|^2\mathrm{e}^{-\gamma t}\right)|\hat{u}_0(\xi)|\notag\\
	&\quad+\chi_{\intt}(\xi)\left(\left(|\xi|^2|\cos(\tilde{\gamma}|\xi|t)|+\frac{|\sin(\tilde{\gamma}|\xi|t)|}{|\xi|}\right)\mathrm{e}^{-\frac{\gamma^2+1}{2\gamma^2}|\xi|^2t}+|\xi|^2\mathrm{e}^{-\gamma t}\right)|\hat{u}_1(\xi)|,
	\end{align}
	where $\tilde{\gamma}:=\sqrt{(\gamma-1)/\gamma}>0$, and
	\begin{align}\label{Point_wise_large}
	\chi_{\extt}(\xi)|\hat{u}(t,\xi)|\lesssim\chi_{\extt}(\xi)\left(\mathrm{e}^{-ct}+\frac{1}{|\xi|^2}\mathrm{e}^{-|\xi|^2t}\right)|\hat{u}_0(\xi)|+\chi_{\extt}(\xi)\frac{1}{|\xi|^2}\left(\mathrm{e}^{-ct}+\mathrm{e}^{-|\xi|^2t}\right)|\hat{u}_1(\xi)|,
	\end{align}
	moreover, it holds that
	\begin{align}
	\chi_{\bdd}(\xi)|\hat{u}(t,\xi)|\lesssim\chi_{\bdd}(\xi)\mathrm{e}^{-ct}(|u_0(\xi)|+|u_1(\xi)|)\label{Point_wise_bdd}
	\end{align}
	for some constants $c>0$.
\end{prop}
\begin{proof} We will divide the proof into three parts with respect to the size of frequencies.\\
	
	\noindent\textbf{Part 1: Estimates  of solutions for small frequencies}\\
	By using the result of asymptotic expansions of eigenvalues, we may derive for $\xi\in\ml{Z}_{\intt}(\varepsilon)$ that
	\begin{align*}
	\widehat{K}_0(t,|\xi|)=\widehat{K}_{0,1}(t,|\xi|)+\widehat{K}_{0,2}(t,|\xi|)+\widehat{K}_{0,3}(t,|\xi|),
	\end{align*}
	where
	\begin{align*}
	\widehat{K}_{0,1/2}(t,|\xi|)&:=\frac{\pm i\sqrt{\gamma(\gamma-1)}|\xi|+\frac{(\gamma-1)^2}{2\gamma}|\xi|^2+\ml{O}(|\xi|^3)}{\pm 2i\sqrt{\gamma(\gamma-1)}|\xi|-\frac{2(\gamma-1)}{\gamma}|\xi|^2+\ml{O}(|\xi|^3)}\mathrm{e}^{\pm i\sqrt{\frac{\gamma-1}{\gamma}}|\xi|t-\frac{\gamma^2+1}{2\gamma^2}|\xi|^2t+\ml{O}(|\xi|^3)t},\\
	\widehat{K}_{0,3}(t,|\xi|)&:=\frac{-\frac{1}{\gamma}|\xi|^2+\ml{O}(|\xi|^3)}{(\gamma-\frac{\gamma^2+3}{\gamma^2}|\xi|^2)^2+\frac{\gamma-1}{\gamma}|\xi|^2+\ml{O}(|\xi|^3)}\mathrm{e}^{-\gamma t+\frac{1}{\gamma^2}|\xi|^2t+\ml{O}(|\xi|^3)t}.
	\end{align*}
	Carrying out some direct computations, one gets
	\begin{align*}
	\chi_{\intt}(\xi)|\widehat{K}_0(t,|\xi|)|\lesssim\chi_{\intt}(\xi)\left(\left|\cos\left(\sqrt{\tfrac{\gamma-1}{\gamma}}|\xi|t\right)\right|+|\xi|\left|\sin\left(\sqrt{\tfrac{\gamma-1}{\gamma}}|\xi|t\right)\right|\right)\mathrm{e}^{-\frac{\gamma^2+1}{2\gamma^2}|\xi|^2t}+\chi_{\intt}(\xi)|\xi|^2\mathrm{e}^{-\gamma t}.
	\end{align*}
	Secondly, we also can decompose the kernel by
	\begin{align*}
	\widehat{K}_1(t,|\xi|)=\widehat{K}_{1,1}(t,|\xi|)+\widehat{K}_{1,2}(t,|\xi|)+\widehat{K}_{1,3}(t,|\xi|),
	\end{align*}
	with the following components:
	\begin{align*}
	\widehat{K}_{1,1/2}(t,|\xi|)&:=\frac{\gamma\pm i\sqrt{\frac{\gamma-1}{\gamma}}|\xi|-\frac{\gamma^2+1}{2\gamma^2}|\xi|^2+\ml{O}(|\xi|^3)}{\pm 2i\sqrt{\gamma(\gamma-1)}|\xi|-\frac{2(\gamma-1)}{\gamma}|\xi|^2+\ml{O}(|\xi|^3)}\mathrm{e}^{\pm i\sqrt{\frac{\gamma-1}{\gamma}}|\xi|t-\frac{\gamma^2+1}{2\gamma^2}|\xi|^2t+\ml{O}(|\xi|^3)t},\\
	\widehat{K}_{1,3}(t,|\xi|)&:=\frac{\frac{1}{\gamma^2}|\xi|^2+\ml{O}(|\xi|^3)}{(\gamma-\frac{\gamma^2+3}{\gamma^2}|\xi|^2)^2+\frac{\gamma-1}{\gamma}|\xi|^2+\ml{O}(|\xi|^3)}\mathrm{e}^{-\gamma t+\frac{1}{\gamma^2}|\xi|^2t+\ml{O}(|\xi|^3)t}.
	\end{align*}
	Thus, the estimate holds
	\begin{align*}
	\chi_{\intt}(\xi)|\widehat{K}_1(t,|\xi|)|\lesssim\chi_{\intt}(\xi)\left(|\xi|^2\left|\cos\left(\sqrt{\tfrac{\gamma-1}{\gamma}}|\xi|t\right)\right|+\frac{\left|\sin\left(\sqrt{\tfrac{\gamma-1}{\gamma}}|\xi|t\right)\right|}{|\xi|}\right)\mathrm{e}^{-\frac{\gamma^2+1}{2\gamma^2}|\xi|^2t}+\chi_{\intt}(\xi)|\xi|^2\mathrm{e}^{-\gamma t}.
	\end{align*}
	Finally, by applying the representation \eqref{Represent_Solution}, we arrive at \eqref{Point_wise_small}.\\
	
	\noindent\textbf{Part 2: Estimates  of solutions for large frequencies}\\
	At this time, the principal parts of the eigenvalues are real number. For this reason, we may easy to deduce the following estimates:
	\begin{align*}
	\chi_{\extt}(\xi)|\widehat{K}_0(t,|\xi|)|&\lesssim\chi_{\extt}(\xi)\left(\mathrm{e}^{-\frac{1+\gamma+\sqrt{(1+\gamma)^2-4(\gamma-1)}}{2}t}+\mathrm{e}^{-\frac{1+\gamma-\sqrt{(1+\gamma)^2-4(\gamma-1)}}{2}t}+\frac{1}{|\xi|^2}\mathrm{e}^{-|\xi|^2t}\right),\\
	\chi_{\extt}(\xi)|\widehat{K}_1(t,|\xi|)|&\lesssim\chi_{\extt}(\xi)\frac{1}{|\xi|^2}\left(\mathrm{e}^{-\frac{1+\gamma+\sqrt{(1+\gamma)^2-4(\gamma-1)}}{2}t}+\mathrm{e}^{-\frac{1+\gamma-\sqrt{(1+\gamma)^2-4(\gamma-1)}}{2}t}+\mathrm{e}^{-|\xi|^2t}\right).
	\end{align*}
	Again, combining the above two estimates with \eqref{Represent_Solution}, we get our desired estimate \eqref{Point_wise_large}.\\
	
	\noindent\textbf{Part 3: Estimates of solutions for bounded frequencies}\\
	According to our conclusion \eqref{BDD} and the fact that $\varepsilon\leqslant|\xi|\leqslant N$, in the cases for multiple root or non-multiple root, we can estimate
	\begin{align*}
	\chi_{\bdd}(\xi)|\widehat{K}_0(t,|\xi|)|+\chi_{\bdd}(\xi)|\widehat{K}_1(t,|\xi|)|\lesssim \chi_{\bdd}(\xi)\mathrm{e}^{-ct}
	\end{align*}
	for some positive constants $c$. It allows us complete the proof.
\end{proof}

\begin{prop}\label{Prop_Pointwise_Est_2}
The solution $\hat{u}=\hat{u}(t,\xi)$ to the initial value problem \eqref{Eq_02} with the relaxation function \eqref{Exponential_Decay_Relax} fulfills the following estimates:
	\begin{align*}
	\chi_{\intt}(\xi)|\hat{u}_t(t,\xi)|&\lesssim\chi_{\intt}(\xi)\left(\left(|\xi|^2|\cos(\tilde{\gamma}|\xi|t)|+|\xi|\,|\sin(\tilde{\gamma}|\xi|t)|\right)\mathrm{e}^{-\frac{\gamma^2+1}{2\gamma^2}|\xi|^2t}+|\xi|^2\mathrm{e}^{-\gamma t}\right)|\hat{u}_0(\xi)|\notag\\
	&\quad+\chi_{\intt}(\xi)\left(\left(|\cos(\tilde{\gamma}|\xi|t)|+|\xi|^3|\sin(\tilde{\gamma}|\xi|t)|\right)\mathrm{e}^{-\frac{\gamma^2+1}{2\gamma^2}|\xi|^2t}+|\xi|^2\mathrm{e}^{-\gamma t}\right)|\hat{u}_1(\xi)|,
	\end{align*}
	where $\tilde{\gamma}:=\sqrt{(\gamma-1)/\gamma}>0$, and
	\begin{align*}
	\chi_{\extt}(\xi)|\hat{u}_t(t,\xi)|\lesssim\chi_{\extt}(\xi)\left(\mathrm{e}^{-ct}+\mathrm{e}^{-|\xi|^2t}\right)|\hat{u}_0(\xi)|+\chi_{\extt}(\xi)\frac{1}{|\xi|^2}\left(\mathrm{e}^{-ct}+|\xi|^2\mathrm{e}^{-|\xi|^2t}\right)|\hat{u}_1(\xi)|,
	\end{align*}
	moreover, it holds that
	\begin{align*}
	\chi_{\bdd}(\xi)|\hat{u}_t(t,\xi)|\lesssim\chi_{\bdd}(\xi)\mathrm{e}^{-ct}(|u_0(\xi)|+|u_1(\xi)|)
	\end{align*}
	for some constants $c>0$.
\end{prop}
\begin{proof}
	By following the similar procedures to those in Proposition \ref{Prop_Pointwise_Est} and using
	\begin{align*}
	\partial_t\widehat{K}_j(t,|\xi|)=\sum\limits_{k=1,2,3}\lambda_k(|\xi|)\widehat{K}_{j,k}(t,|\xi|)
	\end{align*}
	for $j=0,1$, the desired estimates can be proved.
\end{proof}


\section{Asymptotic profiles in a framework of weighted $L^1$ space}\label{Sec_Exp_Asy_Prof}
In this section, we will study asymptotic profiles of solutions to the viscoelastic damped wave equation \eqref{VDW_Equation} with weighted $L^1$ data, where, the weighted $L^1$ space is defined by
\begin{align*}
L^{1,1}(\mb{R}^n):=\left\{f\in L^1(\mb{R}^n):\ \|f\|_{L^{1,1}(\mb{R}^n)}:=\int_{\mb{R}^n}(1+|x|)|f(x)|\mathrm{d}x<\infty\right\}.
\end{align*}
Moreover, the notation of integral for a function is $P_f:=\int_{\mb{R}^n}f(x)\mathrm{d}x$. Therefore, one observes that
\begin{align}\label{Pf_L11}
|P_f|\leqslant\int_{\mb{R}^n}|f(x)|\mathrm{d}x\leqslant\|f\|_{L^{1,1}(\mb{R}^n)}.
\end{align}
Precisely, we will deduce some estimates of upper bounds for solutions in Subsection \ref{Sub_Sec_UPPER} firstly. To show the optimality of the derived result, by using Fourier analysis, we estimate lower bounds for solution in Subsection \ref{Sub_Sec_LOWER}. In the procedure, the second-order expansion of the solution to the Cauchy problem \eqref{VDW_Equation} will be given.

\subsection{Estimates of upper bounds for solutions}\label{Sub_Sec_UPPER}
Before deducing some estimates of solutions, let us introduce some useful estimates in the $L^2$ norm, whose proof can be found in Theorem 2.2 of \cite{C.-Ike-2020}.
\begin{lemma}\label{Prop_Kernel_Est}
	Let $s\geqslant0$, $\gamma>1$ and $\tilde{\gamma}>0$. The following estimates hold:
	\begin{align*}
	\left\|\chi_{\intt}(\xi)|\xi|^s|\cos(\tilde{\gamma}|\xi|t)|\mathrm{e}^{-\frac{\gamma^2+1}{2\gamma^2}|\xi|^2t}\right\|_{L^2(\mb{R}^n)}&\lesssim(1+t)^{-\frac{s}{2}-\frac{n}{4}},\\
	\left\|\chi_{\intt}(\xi)|\xi|^{s-1}|\sin(\tilde{\gamma}|\xi|t)|\mathrm{e}^{-\frac{\gamma^2+1}{2\gamma^2}|\xi|^2t}\right\|_{L^2(\mb{R}^n)}&\lesssim
	\ml{G}(t;s,n),
	\end{align*}
	where the time-dependent function is defined by
	\begin{align*}
	\ml{G}(t;s,n):=\begin{cases}
	(1+t)^{1-s-\frac{n}{2}}&\mbox{if}\ \ 2s+n<2,\\
	(\ln(\mathrm{e}+t))^{\frac{1}{2}}&\mbox{if}\ \ 2s+n=2,\\
	(1+t)^{1-\frac{5s}{6}-\frac{5n}{12}}&\mbox{if}\  \ 2<2s+n<3,\\
	(1+t)^{\frac{1}{2}-\frac{s}{2}-\frac{n}{4}}&\mbox{if} \ \ 2s+n\geqslant 3.
	\end{cases}
	\end{align*}
\end{lemma}
\begin{remark}
Indeed, the first three cases in the function $\ml{G}(t;s,n)$ consider the lower dimensions only. To be specific, the first case includes $n=1$ with $s\in(0,1/2)$; the second case includes $n=1$ with $s=1/2$ and $n=2$ with $s=0$; the third case includes $n=1$ with $s\in(1/2,1)$ and $n=2$ with $s\in(0,1/2)$. The main reason for the consideration of these cases is the singularities for the integral
\begin{align*}
\int_0^{\varepsilon}y^{2s+n-3}|\sin(yt)|^2\mathrm{e}^{-y^2t}\mathrm{d}y
\end{align*}
for $2s+n-3<0$ as $y\downarrow0$.
\end{remark}

\begin{theorem}\label{Thm_Estimates}
	 Let us assume
	\begin{align*}
	(u_0,u_1)\in\ml{D}_s(\mb{R}^n):=(H^s(\mb{R}^n)\cap L^{1,1}(\mb{R}^n))\times (H^{\max\{s-2,0\}}(\mb{R}^n)\cap L^{1,1}(\mb{R}^n)),
	\end{align*}
	with $s\geqslant0$. Then, the solution $u=u(t,x)$ to the Cauchy problem \eqref{VDW_Equation} with the relaxation function \eqref{Exponential_Decay_Relax} fulfills the following estimates:
	\begin{align*}
	\|\,|D|^su(t,\cdot)\|_{L^2(\mb{R}^n)}&\lesssim(1+t)^{-\frac{s+1}{2}-\frac{n}{4}}\|u_0\|_{H^s(\mb{R}^n)\cap L^{1,1}(\mb{R}^n)}+(1+t)^{-\frac{s}{2}-\frac{n}{4}}|P_{u_0}|\\
	&\quad+(1+t)^{-\frac{s}{2}-\frac{n}{4}}\|u_1\|_{H^{\max\{s-2,0\}}(\mb{R}^n)\cap L^{1,1}(\mb{R}^n)}+\ml{G}(t;s,n)|P_{u_1}|.
	\end{align*}
\end{theorem}
\begin{remark}\label{Remark_4.2}
	Providing that $|P_{u_0}|=|P_{u_1}|=0$, e.g. $u_0(x)$ and $u_1(x)$ are odd functions with respect to $x_n$, then we can improve the decay estimates in Theorem \ref{Thm_Estimates} by
	\begin{align*}
	\|\,|D|^su(t,\cdot)\|_{L^2(\mb{R}^n)}\lesssim(1+t)^{-\frac{s+1}{2}-\frac{n}{4}}\|u_0\|_{H^s(\mb{R}^n)\cap L^{1,1}(\mb{R}^n)}+(1+t)^{-\frac{s}{2}-\frac{n}{4}}\|u_1\|_{H^{\max\{s-2,0\}}(\mb{R}^n)\cap L^{1,1}(\mb{R}^n)}
	\end{align*}
	for any $s\geqslant0$ and $n\geqslant 1$.
\end{remark}
\begin{remark}
	Actually, the additional memory term $g\ast\Delta u$ with an exponential decay relaxation function \eqref{Exponential_Decay_Relax} does not influence on the dissipative structure of the viscoelastic damped wave equation. Especially, the estimate of the solution itself stated in Theorem \ref{Thm_Estimates} with $s=0$ is the same as those in the recent paper \cite{Ike-2014}.
\end{remark}
\begin{proof}
	To begin with the estimate, we apply Lemma 2.1 from \cite{Ike-2004} and Proposition \ref{Prop_Pointwise_Est} for $\xi\in\ml{Z}_{\intt}(\varepsilon)$ to investigate
	\begin{align*}
	\chi_{\intt}(\xi)|\xi|^s|\hat{u}(t,\xi)|&\lesssim\chi_{\intt}(\xi)\left(|\xi|^{s+1}|\cos(\tilde{\gamma}|\xi|t)+|\xi|^{s+2}|\sin(\tilde{\gamma}|\xi|t)|\right)\mathrm{e}^{-\frac{\gamma^2+1}{2\gamma^2}|\xi|^2t}\|u_0\|_{L^{1,1}(\mb{R}^n)}\\
	&\quad+\chi_{\intt}(\xi)\left(|\xi|^{s}|\cos(\tilde{\gamma}|\xi|t)+|\xi|^{s+1}|\sin(\tilde{\gamma}|\xi|t)|\right)\mathrm{e}^{-\frac{\gamma^2+1}{2\gamma^2}|\xi|^2t}|P_{u_0}|\\
	&\quad+\chi_{\intt}(\xi)\left(|\xi|^{s+3}|\cos(\tilde{\gamma}|\xi|t)|+|\xi|^s|\sin(\tilde{\gamma}|\xi|t)|\right)\mathrm{e}^{-\frac{\gamma^2+1}{2\gamma^2}|\xi|^2t}\|u_1\|_{L^{1,1}(\mb{R}^n)}\\
	&\quad+\chi_{\intt}(\xi)\left(|\xi|^{s_2}|\cos(\tilde{\gamma}|\xi|t)|+|\xi|^{s-1}|\sin(\tilde{\gamma}|\xi|t)|\right)\mathrm{e}^{-\frac{\gamma^2+1}{2\gamma^2}|\xi|^2t}|P_{u_1}|\\
	&\quad+\mathrm{e}^{-\gamma t}\left(\|u_0\|_{L^{1,1}(\mb{R}^n)}+|P_{u_0}|+\|u_1\|_{L^{1,1}(\mb{R}^n)}+|P_{u_1}|\right).
	\end{align*}
	Then, taking account of Lemma \ref{Prop_Kernel_Est} and H\"older's inequality, we obtain
	\begin{align}\label{Ineq_01}
	\|\chi_{\intt}(D)|D|^su(t,\cdot)\|_{L^2(\mb{R}^n)}&\lesssim(1+t)^{-\frac{s+1}{2}-\frac{n}{4}}\|u_0\|_{L^{1,1}(\mb{R}^n)}+(1+t)^{-\frac{s}{2}-\frac{n}{4}}|P_{u_0}|\notag\\
	&\quad+(1+t)^{-\frac{s}{2}-\frac{n}{4}}\|u_1\|_{L^{1,1}(\mb{R}^n)}+\ml{G}(t;s,n)|P_{u_1}|\notag\\
	&\lesssim (1+t)^{-\frac{s}{2}-\frac{n}{4}}\|u_0\|_{L^{1,1}(\mb{R}^n)}+\ml{G}(t;s,n)\|u_1\|_{L^{1,1}(\mb{R}^n)}.
	\end{align}
	Next, considering the estimates for large frequencies, one has
	\begin{align*}
	\chi_{\extt}(\xi)|\xi|^s|\hat{u}(t,\xi)|\lesssim\chi_{\extt}(\xi)\mathrm{e}^{-ct}\left(|\xi|^s|\hat{u}_0(\xi)|+|\xi|^{s-2}|\hat{u}_1(\xi)|\right),
	\end{align*}
	which implies immediately
	\begin{align*}
	\|\chi_{\extt}(D)|D|^su(t,\cdot)\|_{L^2(\mb{R}^n)}\lesssim \mathrm{e}^{-ct}\|u_0\|_{H^s(\mb{R}^n)}+\mathrm{e}^{-ct}\|u_1\|_{H^{s-2}(\mb{R}^n)}.
	\end{align*}
	Particularly, it holds $H^{s-2}(\mb{R}^n)\subseteq H^{\max\{s-2,0\}}(\mb{R}^n)$ for any $s\geqslant0$. Finally, from the derived estimate \eqref{Point_wise_bdd}, an exponential decay estimate can be found by assuming $L^2$ regularity for initial data. Thus, the proof is complete.
\end{proof}

By processing the similar computation to the Theorem \ref{Thm_Estimates}, we can prove the next result.
\begin{theorem}\label{Thm_Estimates_Deriv}
	Let us assume
	\begin{align*}
	(u_0,u_1)\in\widetilde{\ml{D}}_s(\mb{R}^n):=(H^s(\mb{R}^n)\cap L^{1,1}(\mb{R}^n))\times (H^s(\mb{R}^n)\cap L^{1,1}(\mb{R}^n)),
	\end{align*}
	with $s\geqslant0$. Then, the derivative of the solution $u=u(t,x)$ to the Cauchy problem \eqref{VDW_Equation} with the relaxation function \eqref{Exponential_Decay_Relax} fulfills the following estimates:
	\begin{align*}
	\|\,|D|^su_t(t,\cdot)\|_{L^2(\mb{R}^n)}&\lesssim(1+t)^{-\frac{s+2}{2}-\frac{n}{4}}\|u_0\|_{H^s(\mb{R}^n)\cap L^{1,1}(\mb{R}^n)}+(1+t)^{-\frac{s+1}{2}-\frac{n}{4}}|P_{u_0}|\\
	&\quad+(1+t)^{-\frac{s+1}{2}-\frac{n}{4}}\|u_1\|_{H^s(\mb{R}^n)\cap L^{1,1}(\mb{R}^n)}+(1+t)^{-\frac{s}{2}-\frac{n}{4}}|P_{u_1}|.
	\end{align*}
\end{theorem}
\begin{remark}
Indeed, the existence result can be shown by the next way: if we assume $(u_0,u_1)\in H^{s+2}(\mb{R}^n)\times H^s(\mb{R}^n)$ for any $s\geqslant0$, then there exists a unique determine Sobolev solution to the Cauchy problem \eqref{VDW_Equation} such that
\begin{align*}
u\in\ml{C}([0,T],H^{s+2}(\mb{R}^n))\cap \ml{C}^1([0,T],H^s(\mb{R}^n))
\end{align*}
for any $T>0$. To prove this statement, we divide it into two steps. For one thing, by using Proposition \ref{Prop_Pointwise_Est} and \ref{Prop_Pointwise_Est_2}, it holds that
\begin{align*}
\hat{u}\in L^{\infty}([0,T],L^{2,s+2}(\mb{R}^n))\ \ \mbox{and}\ \ \hat{u}_t\in L^{\infty}([0,T],L^{2,s}(\mb{R}^n)),
\end{align*}
in which $\hat{f}\in L^{2,s}(\mb{R}^n)$ means $|\xi|^s\hat{f}\in L^2(\mb{R}^n)$. For another, to show the continuity with respect to the time variable, we need to verify
\begin{align*}
\lim\limits_{t_1\to t_2}\left\|\hat{u}(t_1,\cdot)-\hat{u}(t_2,\cdot)\right\|_{L^{2,s+2}(\mb{R}^n)}=0\ \ \mbox{and}\ \ \lim\limits_{t_1\to t_2}\left\|\hat{u}_t(t_1,\cdot)-\hat{u}_t(t_2,\cdot)\right\|_{L^{2,s}(\mb{R}^n)}=0
\end{align*}
for $t_1,t_2\in[0,T]$. At this time, we may complete the proof by using mean value theorem. For example, from \eqref{Represent_Solution} combined with mean value theorem for $t_0\in(t_1,t_2)$, we get
\begin{align*}
\left|\widehat{K}_0(t_1,|\xi|)-\widehat{K}_0(t_2,|\xi|)\right|&\leqslant|t_1-t_2| \sum\limits_{j=1,2,3}\left|\frac{\prod\nolimits_{k=1,2,3,\ k\neq j}\lambda_k(|\xi|)-|\xi|^2}{\prod\nolimits_{k=1,2,3,\ k\neq j}(\lambda_j(|\xi|)-\lambda_k(|\xi|))}\right||\lambda_j(|\xi|)|\mathrm{e}^{\lambda_j(|\xi|)t_0}.
\end{align*}
To estimate the right-hand sides in the $L^2$ norm, the procedure is similar to those for estimating $\hat{u}_t(t,\cdot)$ in the $L^2$ norm without any technical difficulties. Then, our statement can be verified.
\end{remark}

\subsection{Estimates of lower bounds for solutions}\label{Sub_Sec_LOWER}
Let us find the second-order expansion of the solution in the first place, which will contribute to lower bound estimates of the solution. Indeed, the crucial part is the construction of leading terms. We now define the leading terms for the kernels $\widehat{K}_{0,1/2}(t,|\xi|)$ by ignoring the higher-order terms including $\ml{O}(|\xi|^3)$ such that
\begin{align}\label{J_0,1/2}
\widehat{J}_{0,1/2}(t,|\xi|):=\frac{\pm i\sqrt{\gamma(\gamma-1)}+\frac{(\gamma-1)^2}{2\gamma}|\xi|}{\pm 2i\sqrt{\gamma(\gamma-1)}-\frac{2(\gamma-1)}{\gamma}|\xi|}\mathrm{e}^{\pm i\sqrt{\frac{\gamma-1}{\gamma}}|\xi|t-\frac{\gamma^2+1}{2\gamma^2}|\xi|^2t},
\end{align}
and analogously for the kernels $\widehat{K}_{1,1/2}(t,|\xi|)$ such that
\begin{align}\label{J_1,1/2}
\widehat{J}_{1,1/2}(t,|\xi|)&:=\frac{\gamma\pm i\sqrt{\frac{\gamma-1}{\gamma}}|\xi|-\frac{\gamma^2+1}{2\gamma^2}|\xi|^2}{\pm 2i\sqrt{\gamma(\gamma-1)}|\xi|-\frac{2(\gamma-1)}{\gamma}|\xi|^2}\mathrm{e}^{\pm i\sqrt{\frac{\gamma-1}{\gamma}}|\xi|t-\frac{\gamma^2+1}{2\gamma^2}|\xi|^2t}.
\end{align}
Moreover, we denote the sum of the previous leading terms by
\begin{align*}
J_j(t,|x|):=\ml{F}^{-1}_{\xi\to x}\left(\widehat{J}_j(t,|\xi|)\right):=\ml{F}^{-1}_{\xi\to x}\left(\widehat{J}_{j,1}(t,|\xi|)+\widehat{J}_{j,2}(t,|\xi|)\right)
\end{align*}
for $j=0,1$.

Speaking about the consideration of leading terms, we omitted the examination of the kernels $\widehat{K}_{0,3}(t,|\xi|)$ as well as $\widehat{K}_{1,3}(t,|\xi|)$ because they just exert a small perturbation in the sense of exponential decay type. In other words, they will not influence on the second-order expansion.

Additionally, to show $J_0(t,|x|)$ and $J_1(t,|x|)$ really being the leading terms, we have to derive some error estimates which somehow will provide some gained decay rates. These gained decay rates are strongly related to the generalized diffusion phenomenon (see, for example, \cite{Nishihara2003,Ike-Tod-Yor-2013}).

\begin{theorem}\label{Thm_DIFF}
Let us assume $(u_0,u_1)\in L^{1,1}(\mb{R}^n)\times L^{1,1}(\mb{R}^n)$. Then, the solution $u=u(t,x)$ to the Cauchy problem \eqref{VDW_Equation} with the relaxation function \eqref{Exponential_Decay_Relax} fulfills the following refined estimates:
\begin{align*}
\left\|\chi_{\intt}(D)\left(u(t,\cdot)-\left(J_0(t,|\cdot|)P_{u_0}+J_1(t,|\cdot|)P_{u_1}\right)\right)\right\|_{L^2(\mb{R}^n)}\lesssim t^{-\frac{1}{2}-\frac{n}{4}}\|u_0\|_{L^{1,1}(\mb{R}^n)}+ t^{-\frac{n}{4}}\|u_1\|_{L^{1,1}(\mb{R}^n)}
\end{align*}
for $t\gg1$.
\end{theorem}
\begin{remark}
Comparing with the estimate \eqref{Ineq_01} with $s=0$, we observe that the decay rates stated in Theorem \ref{Thm_DIFF} have been improved for $t\gg1$ if we subtract the functions $J_0(t,|x|)P_{u_0}$ and $J_1(t,|x|)P_{u_1}$. Precisely, concerning $t\gg1$, the decay rate for $u_0(x)$ is improved by $t^{-\frac{1}{2}}$ and the decay rate for $u_1(x)$ is improved by $t^{-\frac{3}{4}}$ if $n=1$, $t^{-\frac{1}{2}}(\ln t)^{-\frac{1}{2}}$ if $n=2$, $t^{-\frac{1}{2}}$ if $n\geqslant 3$. These improvements indeed indicate that $J_{0}(t,|x|)P_{u_0}+J_1(t,|x|)P_{u_1}$ approximates to the solution of the viscoelastic damped wave equation with memory in the $L^2$ norm.
\end{remark}
\begin{remark}
Concerning the viscoelastic damped wave equation without memory, namely,
\begin{align}\label{VDW_Equation_without_Memory}
\begin{cases}
v_{tt}-\Delta v-\Delta v_t=0,&x\in\mb{R}^n,\ t>0,\\
(v,v_t)(0,x)=(v_0,v_1)(x),&x\in\mb{R}^n,
\end{cases}
\end{align}
the recent paper \cite{Ike-2014} found that the asymptotic profile of $v=v(t,x)$ is related to the diffusion waves. Precisely, the profile of solution to \eqref{VDW_Equation_without_Memory} in the $L^2$ norm can be described by
\begin{align*}
\ml{W}_{\mathrm{diff}}(t,x):=\ml{F}^{-1}_{\xi\to x}\left(\cos(|\xi|t)\mathrm{e}^{-\frac{|\xi|^2}{2}t}\right)P_{v_0}+\ml{F}^{-1}_{\xi\to x}\left(\frac{\sin(|\xi|t)}{|\xi|}\mathrm{e}^{-\frac{|\xi|^2}{2}t}\right)P_{v_1}.
\end{align*} 
However, due to the memory effect generated by $g\ast\Delta u$ with the exponential decay relaxation function in \eqref{VDW_Equation}, the profile of $u=u(t,x)$ has been changed into a more complex form. Precisely, according to \eqref{J_0,1/2} and \eqref{J_1,1/2}, the profile of solution to \eqref{VDW_Equation} in the $L^2$ norm can be described by 
\begin{align*}
&J_0(t,|x|)P_{u_0}+J_1(t,|x|)P_{u_1}\\
&=\ml{F}^{-1}_{\xi\to x}\left(\left(\frac{2\gamma^3-(\gamma-1)^2|\xi|^2}{2\gamma^3+2(\gamma-1)|\xi|^2}\cos(\tilde{\gamma}|\xi|t)+\frac{\gamma(\gamma+1)\sqrt{\gamma(\gamma-1)}|\xi|}{2\gamma^3+2(\gamma-1)|\xi|^2}\sin(\tilde{\gamma}|\xi|t)\right)\mathrm{e}^{-\frac{\gamma^2+1}{2\gamma^2}|\xi|^2t}\right)P_{u_0}\\
&\quad+\ml{F}_{\xi\to x}^{-1}\left(\left(\frac{(\gamma^2+1)|\xi|^2}{2\gamma^4+2\gamma(\gamma-1)|\xi|^2}\cos(\tilde{\gamma}|\xi|t)+\frac{2\gamma^3-(\gamma-3)(\gamma+1)|\xi|^2}{2\gamma^3\tilde{\gamma}+2(\gamma-1)\tilde{\gamma}|\xi|^2}\frac{\sin(\tilde{\gamma}|\xi|t)}{|\xi|}\right)\mathrm{e}^{-\frac{\gamma^2+1}{2\gamma^2}|\xi|^2t}\right)P_{u_1}.
\end{align*}
In the view of the asymptotic behaviors, we may find
\begin{align*}
J_0(t,|x|)P_{u_0}+J_1(t,|x|)P_{u_1}\sim\ml{W}_{\mathrm{diff}}(t,x)-\Delta \ml{W}_{\mathrm{diff}}(t,x).
\end{align*}
It means that the exponential decay memory term provides an additional influence on the profiles of solution, i.e. second-order derivative with respect to spatial variables for the diffusion waves. In fact, the function $-\Delta \ml{W}_{\mathrm{diff}}(t,x)$ also can be understood as the solution to the higher-order diffusion waves. Finally, we conjecture a generalized diffusion phenomenon that for $(u_0,u_1)\in\ml{D}_0(\mb{R}^n)$ with $n\geqslant 1$, it holds that
\begin{align*}
\lim\limits_{t\to\infty}\left((\ml{G}(t;0,n))^{-1}\left\|u(t,\cdot)-\ml{M}_1\vec{\ml{w}}(t,\cdot)-\ml{M}_2\Delta\vec{\ml{w}}(t,\cdot)\right\|_{L^2(\mb{R}^n)}\right)=0,
\end{align*}
where $\ml{M}_1,\ml{M}_2\in\mb{C}^{1\times2}$ and $\vec{\ml{w}}=\vec{\ml{w}}(t,x)$ is the vector solution to
\begin{align*}
\vec{\ml{w}}_t-\frac{\gamma^2+1}{2\gamma^2}\mathrm{diag}(1,1)\Delta\vec{\ml{w}}+i\sqrt{\frac{\gamma-1}{\gamma}}\mathrm{diag}(1,-1)(-\Delta)^{1/2}\vec{\ml{w}}=0,
\end{align*}
with suitable initial data $\vec{\ml{w}}(0,x)=\vec{\ml{w}}_0(x)$.
\end{remark}
\begin{proof}
	According to the representation of the solution \eqref{Represent_Solution}, we may observe
	\begin{align*}
	&\chi_{\intt}(\xi)\left|\hat{u}(t,\xi)-\widehat{J}_0(t,|\xi|)\hat{u}_0(\xi)-\widehat{J}_1(t,|\xi|)\hat{u}_1(\xi) \right|\\
	&\lesssim\chi_{\intt}(\xi)\sum\limits_{j=0,1}\sum\limits_{k=1,2}\left|\widehat{K}_{j,k}(t,|\xi|)-\widehat{J}_{j,k}(t,|\xi|)\right||\hat{u}_j(\xi)|+\chi_{\intt}(\xi)\sum\limits_{j=0,1}|\widehat{K}_{j,3}(t,|\xi|)||\hat{u}_j(\xi)|.
	\end{align*}
	For one thing, by direct calculations associated with the formulas
	\begin{align*}
	\frac{\left(f_1(|\xi|)+\ml{O}(|\xi|^3)\right)\mathrm{e}^{\ml{O}(|\xi|^3)t}}{f_2(|\xi|)+\ml{O}(|\xi|^3)}-\frac{f_1(|\xi|)}{f_2(|\xi|)}=\frac{f_1(|\xi|)f_2(|\xi|)\left(\mathrm{e}^{\ml{O}(|\xi|^3)t}-1\right)+\left(f_2(|\xi|)\mathrm{e}^{\ml{O}(|\xi|^3)t}-f_1(|\xi|)\right)\ml{O}(|\xi|^3)}{(f_2(|\xi|))^2+f_2(|\xi|)\ml{O}(|\xi|^3)},
	\end{align*}
	as well as
	\begin{align*}
	\mathrm{e}^{\ml{O}(|\xi|^3)t}-1=\ml{O}(|\xi|^3)t\int_0^1\mathrm{e}^{\ml{O}(|\xi|^3)ts}\mathrm{d}s,
	\end{align*}
	 it holds
	\begin{align}\label{Eq_03}
	&\chi_{\intt}(\xi)\left|\widehat{K}_{0,1/2}(t,|\xi|)-\widehat{J}_{0,1/2}(t,|\xi|)\right|\notag\\
	&\lesssim\chi_{\intt}(\xi)\mathrm{e}^{-\frac{\gamma^2+1}{2\gamma^2}|\xi|^2t}\left|\frac{\left(\pm i\sqrt{\gamma(\gamma-1)}|\xi|+\frac{(\gamma-1)^2}{2\gamma}|\xi|^2+\ml{O}(|\xi|^3)\right)\mathrm{e}^{\ml{O}(|\xi|^3)t}}{\left(\pm 2i\sqrt{\gamma(\gamma-1)}|\xi|-\frac{2(\gamma-1)}{\gamma}|\xi|^2+\ml{O}(|\xi|^3)\right)}-\frac{\pm i\sqrt{\gamma(\gamma-1)}+\frac{(\gamma-1)^2}{2\gamma}|\xi|}{\pm 2i\sqrt{\gamma(\gamma-1)}-\frac{2(\gamma-1)}{\gamma}|\xi|}\right|\notag\\
	&\lesssim \chi_{\intt}(\xi)\mathrm{e}^{-\frac{\gamma^2+1}{2\gamma^2}|\xi|^2t}\left(\ml{O}(|\xi|^3)t\int_0^1\mathrm{e}^{\ml{O}(|\xi|^3)ts}\mathrm{d}s+\ml{O}(|\xi|^2)\right).
	\end{align}
	Similarly, we obtain
	\begin{align}\label{Eq_04}
	&\chi_{\intt}(\xi)\left|\widehat{K}_{1,1/2}(t,|\xi|)-\widehat{J}_{1,1/2}(t,|\xi|)\right|\notag\\
	&\lesssim\chi_{\intt}(\xi)\mathrm{e}^{-\frac{\gamma^2+1}{2\gamma^2}|\xi|^2t}\left|\frac{\left(\gamma\pm i\sqrt{\frac{\gamma-1}{\gamma}}|\xi|-\frac{\gamma^2+1}{2\gamma^2}|\xi|^2+\ml{O}(|\xi|^3)\right)\mathrm{e}^{\ml{O}(|\xi|^3)t}}{\left(\pm 2i\sqrt{\gamma(\gamma-1)}|\xi|-\frac{2(\gamma-1)}{\gamma}|\xi|^2+\ml{O}(|\xi|^3)\right)}-\frac{\gamma\pm i\sqrt{\frac{\gamma-1}{\gamma}}|\xi|-\frac{\gamma^2+1}{2\gamma^2}|\xi|^2}{\pm 2i\sqrt{\gamma(\gamma-1)}|\xi|-\frac{2(\gamma-1)}{\gamma}|\xi|^2}\right|\notag\\
	&\lesssim \chi_{\intt}(\xi)\mathrm{e}^{-\frac{\gamma^2+1}{2\gamma^2}|\xi|^2t}\left(\ml{O}(|\xi|^2)t\int_0^1\mathrm{e}^{\ml{O}(|\xi|^3)ts}\mathrm{d}s+\ml{O}(|\xi|)\right).
	\end{align}
	In other words, we arrive at
	\begin{align*}
	&\chi_{\intt}(\xi)\left|\hat{u}(t,\xi)-\widehat{J}_0(t,|\xi|)\hat{u}_0(\xi)-\widehat{J}_1(t,|\xi|)\hat{u}_1(\xi) \right|\\
	&\lesssim \chi_{\intt}(\xi)\mathrm{e}^{-c|\xi|^2t}\left(\ml{O}(|\xi|^2)t+\ml{O}(|\xi|)\right)\left(|\xi||\hat{u}_0(\xi)|+|\hat{u}_1(\xi)|\right)+\chi_{\intt}(\xi)\mathrm{e}^{-ct}(|\hat{u}_0(\xi)|+|\hat{u}_1(\xi)|).
	\end{align*}
	Next, initial data can be decomposed by
	\begin{align*}
	\hat{u}_j(\xi)=P_{u_j}+A_j(\xi)-iB_j(\xi)
	\end{align*}
	for $j=0,1$, where
	\begin{align*}
	A_j(\xi):=\int_{\mb{R}^n}u_j(x)(1-\cos(x\cdot\xi))\mathrm{d}x\ \ \mbox{and}\ \ B_j(\xi):=\int_{\mb{R}^n}u_j(x)\sin(x\cdot\xi)\mathrm{d}x.
	\end{align*}
	According to Lemma 2.2 in \cite{Ike-2014}, one may have
	\begin{align}\label{Est_A,B}
	|A_j(\xi)|+|B_j(\xi)|\lesssim |\xi|\,\|u_j\|_{L^{1,1}(\mb{R}^n)}
	\end{align}
	for $j=0,1$. It leads to the representation \eqref{Represent_Solution} which can be rewritten by
	\begin{align}\label{Decom}
	\hat{u}(t,\xi)=\sum\limits_{j=0,1}\left(\widehat{K}_j(t,|\xi|)P_{u_j}+\widehat{K}_j(t,|\xi|)(A_j(\xi)-iB_j(\xi))\right).
	\end{align}
	By employing \eqref{Eq_03}, \eqref{Eq_04}, \eqref{Est_A,B} and \eqref{Decom}, the solution can be estimated by
	\begin{align*}
	&\left\|\chi_{\intt}(D)\left(u(t,\cdot)-\left(J_0(t,|\cdot|)P_{u_0}+J_1(t,|\cdot|)P_{u_1}\right)\right)\right\|_{L^2(\mb{R}^n)}\\
	&=\left\|\chi_{\intt}(\xi)\left(\hat{u}(t,\xi)-\left(\widehat{J}_0(t,|\xi|)P_{u_0}+\widehat{J}_1(t,|\xi|)P_{u_1}\right)\right)\right\|_{L^2(\mb{R}^n)}\\
	&\lesssim\left\|\chi_{\intt}(\xi)\left(\widehat{K}_0(t,|\xi|)-\widehat{J}_0(t,|\xi|)\right)\right\|_{L^2(\mb{R}^n)}|P_{u_0}|+\left\|\chi_{\intt}(\xi)\left(\widehat{K}_1(t,|\xi|)-\widehat{J}_1(t,|\xi|)\right)\right\|_{L^2(\mb{R}^n)}|P_{u_1}|\\
	&\quad+\sum\limits_{j=0,1}\left\|\chi_{\intt}(\xi)\widehat{K}_j(t,|\xi|)(A_j(\xi)-iB_j(\xi))\right\|_{L^2(\mb{R}^n)}\\
	&\lesssim \left\|\chi_{\intt}(\xi)(|\xi|^3t+|\xi|^2)\mathrm{e}^{-c|\xi|^2t}\right\|_{L^2(\mb{R}^n)}|P_{u_0}|+\left\|\chi_{\intt}(\xi)(|\xi|^2t+|\xi|)\mathrm{e}^{-c|\xi|^2t}\right\|_{L^2(\mb{R}^n)}|P_{u_1}|\\
	&\quad+\mathrm{e}^{-ct}\left(|P_{u_0}|+|P_{u_1}|\right)+\sum\limits_{j=0,1}\left\|\chi_{\intt}(\xi)|\xi|\widehat{K}_j(t,|\xi|)\right\|_{L^2(\mb{R}^n)}\|u_j\|_{L^{1,1}(\mb{R}^n)}\\
	&\lesssim t^{-\frac{1}{2}-\frac{n}{4}}\|u_0\|_{L^{1,1}(\mb{R}^n)}+t^{-\frac{n}{4}}\|u_1\|_{L^{1,1}(\mb{R}^n)}
	\end{align*}
	for $t\gg1$, where \eqref{Pf_L11} was applied. The proof is complete.
\end{proof}

From Theorem \ref{Thm_Estimates}, in the case $|P_{u_1}|\neq0$, upper bounds of the solution to the Cauchy problem \eqref{VDW_Equation} with the relaxation function \eqref{Exponential_Decay_Relax} can be controlled by
\begin{align*}
\|u(t,\cdot)\|_{L^2(\mb{R}^n)}\lesssim \ml{H}(t;n)\|(u_0,u_1)\|_{\ml{D}_0(\mb{R}^n)},
\end{align*}
for any $t\gg1$, where we used \eqref{Pf_L11} and the time-dependent coefficient is
\begin{align*}
\ml{H}(t;n):=\begin{cases}
t^{\frac{1}{2}}&\mbox{if}\ \ n=1,\\
(\ln t)^{\frac{1}{2}}&\mbox{if} \ \ n=2,\\
t^{\frac{1}{2}-\frac{n}{4}}&\mbox{if} \ \ n\geqslant 3.
\end{cases}
\end{align*}
Here, we should underline that $\ml{G}(t;0,n)\approx \ml{H}(t;n)$ for $t\gg1$. For this reason, the natural question is the optimality of the previous estimate. Namely, sharp lower bound estimates for $\|u(t,\cdot)\|_{L^2(\mb{R}^n)}$ as $t\gg1$. To answer this question, let us recall a useful lemma whose proofs were shown in \cite{Ike-2014,Ike-Ono-2017}.
\begin{lemma}\label{Prop_Kernel_Est_Lower}
	Let $\gamma>1$ and $\tilde{\gamma}>0$. The following estimates hold:
	\begin{align*}
	\left\|\chi_{\intt}(\xi)|\cos(\tilde{\gamma}|\xi|t)|\mathrm{e}^{-\frac{\gamma^2+1}{2\gamma^2}|\xi|^2t}\right\|_{L^2(\mb{R}^n)}&\gtrsim t^{-\frac{n}{4}},\\
	\left\|\chi_{\intt}(\xi)|\xi|^{-1}|\sin(\tilde{\gamma}|\xi|t)|\mathrm{e}^{-\frac{\gamma^2+1}{2\gamma^2}|\xi|^2t}\right\|_{L^2(\mb{R}^n)}&\gtrsim \ml{H}(t;n),
	\end{align*}
for any $t\gg1$.
\end{lemma}

\begin{theorem}\label{Thm_Asym}
	Let us assume $(u_0,u_1)\in \ml{D}_0(\mb{R}^n)$ carrying $|P_{u_1}|\neq 0$. Then, the solution $u=u(t,x)$ to the Cauchy problem \eqref{VDW_Equation} with the relaxation function \eqref{Exponential_Decay_Relax} fulfills the following estimates:
	\begin{align*}
	\ml{H}(t;n)|P_{u_1}|\lesssim \|u(t,\cdot)\|_{L^2(\mb{R}^n)}\lesssim \ml{H}(t;n)\|(u_0,u_1)\|_{\ml{D}_0(\mb{R}^n)}
	\end{align*}
	for any $t\gg1$.
\end{theorem}
\begin{remark}
We observe from Theorem \ref{Thm_Asym} that for $t\gg1$, the upper bound and lower bound estimates for the solution itself in the $L^2$ norm are exactly the same provided $|P_{u_1}|\neq 0$. According to the fact \eqref{Pf_L11}, it means that the derived estimates in Theorem \ref{Thm_Asym} are optimal for all $n\geqslant 1$, i.e. $\|u(t,\cdot)\|_{L^2(\mb{R}^n)}\sim\ml{H}(t;n)$ provided $|P_{u_1}|\neq0$.
\end{remark}
\begin{proof}
	The upper bound estimate for $\|u(t,\cdot)\|_{L^2(\mb{R}^n)}$ is a trivial conclusion from Theorem \ref{Thm_Estimates}. Let us just focus on the lower bound estimate. With the help of the Minkowski inequality, we derive
	\begin{align*}
	\|\chi_{\intt}(D)u(t,\cdot)\|_{L^2(\mb{R}^n)}&\geqslant \|\chi_{\intt}(D)J_0(t,|\cdot|)\|_{L^2(\mb{R}^n)}|P_{u_0}|+\|\chi_{\intt}(D)J_1(t,|\cdot|)\|_{L^2(\mb{R}^n)}|P_{u_1}|\\
	&\quad-\left\|\chi_{\intt}(D)\left(u(t,\cdot)-\left(J_0(t,|\cdot|)P_{u_0}+J_1(t,|\cdot|)P_{u_1}\right)\right)\right\|_{L^2(\mb{R}^n)}.
	\end{align*}
	Concerning the first term on the right-hand side of the above inequality, it is nonnegative. Let us estimate the second term on the right-hand side. Applying the Parseval equality, one derives
	\begin{align*}
	\|\chi_{\intt}(D)J_1(t,|\cdot|)\|_{L^2(\mb{R}^n)}&=\left\|\chi_{\intt}(\xi)\left(\widehat{J}_{1,1}(t,|\xi|)+\widehat{J}_{1,2}(t,|\xi|)\right)\right\|_{L^2(\mb{R}^n)}\\
	&=\left\|\chi_{\intt}(\xi)\frac{(\gamma^2+1)|\xi|^2}{2\gamma^4+2\gamma(\gamma-1)|\xi|^2}\cos(\tilde{\gamma}|\xi|t)\mathrm{e}^{-\frac{\gamma^2+1}{2\gamma^2}|\xi|^2t}\right.\\
	&\quad\quad\left.+\chi_{\intt}(\xi)\frac{2\gamma^3-(\gamma-3)(\gamma+1)|\xi|^2}{2\gamma^3\tilde{\gamma}+2(\gamma-1)\tilde{\gamma}|\xi|^2}\frac{\sin(\tilde{\gamma}|\xi|t)}{|\xi|}\mathrm{e}^{-\frac{\gamma^2+1}{2\gamma^2}|\xi|^2t}\right\|_{L^2(\mb{R}^n)}.
	\end{align*}
	Thus, by using Lemmas \ref{Prop_Pointwise_Est} and \ref{Prop_Kernel_Est_Lower}, it yields
	\begin{align*}
	&\|\chi_{\intt}(D)J_1(t,|\cdot|)\|_{L^2(\mb{R}^n)}\\
	&\gtrsim\left|\left\|\chi_{\intt}(\xi)|\xi|^{-1}|\sin(\tilde{\gamma}|\xi|t)|\mathrm{e}^{-\frac{\gamma^2+1}{2\gamma^2}|\xi|^2t}\right\|_{L^2(\mb{R}^n)}-\left\|\chi_{\intt}(\xi)|\xi|^2|\cos(\tilde{\gamma}|\xi|t)|\mathrm{e}^{-\frac{\gamma^2+1}{2\gamma^2}|\xi|^2t}\right\|_{L^2(\mb{R}^n)}\right|\\
	&\gtrsim\left|\ml{H}(t;n)-t^{-1-\frac{n}{4}}\right|\gtrsim\ml{H}(t;n)
	\end{align*}
	for any $t\gg1$. Finally, by combining the previous derived inequalities and Theorem \ref{Thm_DIFF}, we claim
	\begin{align*}
	\|u(t,\cdot)\|_{L^2(\mb{R}^n)}&\gtrsim \|\chi_{\intt}(D)u(t,\cdot)\|_{L^2(\mb{R}^n)}\\
	&\gtrsim \ml{H}(t;n)|P_{u_1}|-\left(t^{-\frac{1}{2}-\frac{n}{4}}\|u_0\|_{L^{1,1}(\mb{R}^n)}+ t^{-\frac{n}{4}}\|u_1\|_{L^{1,1}(\mb{R}^n)}\right)\\
	&\gtrsim \ml{H}(t;n)|P_{u_1}|
	\end{align*}
	for any $t\gg1$, where we used $\ml{H}(t;n)\gg t^{-\frac{n}{4}}$ for large time. So, the proof of the desired theorem is complete.
	\end{proof}

\section{Singular limit problem}\label{Sec_Sing_Lim}
We introduce the following Cauchy problem for the singular limit problem for the MGT equation with memory in the dissipative case:
\begin{align}\label{Sing_Lim_MGT}
\begin{cases}
\tau v_{\tau,ttt}+v_{\tau,tt}-\Delta v_{\tau}-\Delta v_{\tau,t}+g\ast\Delta v_{\tau}=0,&x\in\mb{R}^n,\ t>0,\\
(v_{\tau},v_{\tau,t},v_{\tau,tt})(0,x)=(v_0,v_1,v_2)(x),&x\in\mb{R}^n,
\end{cases}
\end{align}
with $\tau\in(0,1)$ and initial data $v_0(x)=u_0(x)$, $v_1(x)=u_1(x)$ chosen in the viscoelastic damped wave equation \eqref{VDW_Equation}, where the derivative for $v_{\tau}=v_{\tau}(t,x)$ with respect to the time variable is denoted by $v_{\tau,t}:=\partial_tv_{\tau}$ and similarly for $v_{\tau,tt}$ as well as $v_{\tau,ttt}$. The relaxation function $g=g(t)$ is an exponential decay function defined in \eqref{Exponential_Decay_Relax}. Here, we assume $u_0(x)$ and $u_1(x)$ are nontrivial simultaneously which are useful to guarantee a nontrivial solution to \eqref{VDW_Equation}.

We should remark that the well-posedness and some decay properties to the Cauchy problem \eqref{Sing_Lim_MGT} for each $\tau\in(0,1]$ were given by the recent paper \cite{Bounadja-SaidHouari2020}, in which the ansatz with respect to the auxiliary past history variables plays an important role.

\begin{remark}
Concerning the characteristic equation for the MGT equation \eqref{Sing_Lim_MGT} with the exponential decay relaxation function \eqref{Exponential_Decay_Relax}, we may write down its characteristic equation by following the approach in Section \ref{Sec_Asy_Beh_Solution} as follows:
\begin{align*}
\tau \mu^4+(1+\tau\gamma)\mu^3+\left(|\xi|^2+\gamma\right)\mu^2+(1+\gamma)|\xi|^2\mu+(\gamma-1)|\xi|^2=0.
\end{align*}
It has the roots $\mu_j=\mu_j(|\xi|)$ for $j=1,2,3,4$, which can be asymptotically expanded as
\begin{align*}
\mu_{1,2}(|\xi|)&=\pm i\sqrt{\frac{\gamma-1}{\gamma}}|\xi|-\frac{(1-\tau)\gamma^2+\tau\gamma+1}{2\gamma^2}|\xi|^2+\ml{O}(|\xi|^3),\\
\mu_3(|\xi|)&=-\frac{1}{\tau}+\ml{O}(|\xi|),\ \ \mu_4(|\xi|)=-\gamma+\ml{O}(|\xi|),
\end{align*}
for $\xi\in\ml{Z}_{\intt}(\varepsilon)$, and
\begin{align*}
\mu_{1,2}(|\xi|)&=-\frac{1+\gamma\pm\sqrt{(1+\gamma)^2-4(\gamma-1)}}{2}+\ml{O}(|\xi|^{-1}),\\
\mu_{3,4}(|\xi|)&=\pm i\sqrt{\frac{1}{\tau}}|\xi|-\frac{1-\tau}{2\tau}+\ml{O}(|\xi|^{-1}),
\end{align*}
for $\xi\in\ml{Z}_{\extt}(N)$. Comparing the results in Section \ref{Sec_Asy_Beh_Solution} with the above asymptotic behaviors, we may conjecture the dissipative properties for the viscoelastic damped wave \eqref{VDW_Equation} and the MGT equation \eqref{Sing_Lim_MGT} are similar if the memory term owns the exponential decay relaxation function.
\end{remark}

Our aim in this section is to understand the convergence of solutions for the MGT equation in the dissipative case \eqref{Sing_Lim_MGT} to those for the viscoelastic damped wave equation \eqref{VDW_Equation} as $\tau\downarrow0$. In particular, we are interested in the rate of convergence under different assumptions of initial data and different norms.

\subsection{Singular limit relation on standard energy}
First of all, we act $\tau\partial_t$ on the equation in \eqref{VDW_Equation} and then add the resultant with the equation itself. Consequently, a kind of inhomogeneous MGT equation comes as follows:
\begin{align*}
\tau u_{ttt}+u_{tt}-\Delta u-\Delta u_t+g\ast \Delta u=F(u):=\tau \Delta u_t+\tau\Delta u_{tt}-\tau\Delta u+\gamma\tau g\ast\Delta u.
\end{align*}
Let us denote the difference between the MGT equation with memory and the viscoelastic damped wave equation with memory \eqref{Exponential_Decay_Relax} such that
\begin{align}\label{Diff}
w(t,x):=v_{\tau}(t,x)-u(t,x),
\end{align}
which satisfies the Cauchy problem
\begin{align}\label{Eq_Sing_Lim_Energy}
\begin{cases}
\tau w_{ttt}+w_{tt}-\Delta w-\Delta w_t+g\ast \Delta w=-F(u),&x\in\mb{R}^n,\ t>0,\\
(w,w_t,w_{tt})(0,x)=(0,0,w_2)(x),&x\in\mb{R}^n,
\end{cases}
\end{align}
where we denoted $w_2(x):=v_2(x)-(\Delta u_0(x)+\Delta u_1(x))$.
\begin{theorem}\label{Thm_Sing_Lim_Energy} Let us assume $(u_0,u_1)\in \widetilde{\ml{D}}_4(\mb{R}^n)$ and $v_2\in L^2(\mb{R}^n)$, where $u_0$ and $u_1$ are not zero simultaneously. Then, standard energy for the difference $w=w(t,x)$ defined in \eqref{Diff} fulfills the following estimate for $0<\tau\ll 1$: 
	\begin{align}\label{Est_Sing_Lim_Ener}
	&\tau\|w_{tt}(t,\cdot)\|_{L^2(\mb{R}^n)}^2+\|\nabla w_t(t,\cdot)\|_{L^2(\mb{R}^n)}^2+\|\nabla w(t,\cdot)\|_{L^2(\mb{R}^n)}^2+\tau\|w_t(t,\cdot)\|_{L^2(\mb{R}^n)}^2\notag\\
	&+\gamma\int_0^tg(t-s)\|\nabla w(t,\cdot)-\nabla w(s,\cdot)\|_{L^2(\mb{R}^n)}^2\mathrm{d}s\notag\\
	&\leqslant C\tau\|v_2-\Delta u_0-\Delta u_1\|_{L^2(\mb{R}^n)}^2+C\tau^2\|u_0\|_{H^4(\mb{R}^n)\cap L^{1,1}(\mb{R}^n)}^2+C\tau^2\kappa_n(t)\|u_1\|^2_{H^4(\mb{R}^n)\cap L^{1,1}(\mb{R}^n)},
	\end{align}
	where the time-dependent coefficient is defined by
	\begin{align}\label{kappa}
	\kappa_n(t):=\begin{cases}
	(1+t)^{\frac{1}{2}}&\mbox{if}\ \ n=1,\\
	\ln(\mathrm{e}+t)&\mbox{if}\ \ n=2,\\
	1&\mbox{if}\ \ n\geqslant 3.
	\end{cases}
	\end{align}
	Additionally, if $|P_{u_1}|=0$, then the coefficient $\kappa_n(t)\equiv1$ in \eqref{Est_Sing_Lim_Ener} for any $n\geqslant 1$.
\end{theorem}
\begin{remark}
If we assume additionally $|P_{u_1}|=0$ as the last statement in Theorem \ref{Thm_Sing_Lim_Energy}, according to Remark \ref{Remark_4.2}, the global (in time) convergence result can be extended to all $u_1$ with $n\geqslant 1$. The main change is the additional decay function $(1+t)^{-1}$ localizing on the time-dependent function of the norm for $u_1$ in \eqref{Est_01} as well as \eqref{Est_02}. Since $(1+t)^{-1-\frac{n}{2}}\in L^1((0,\infty))$ for any $n\geqslant 1$, the above statement can be proved easily.
\end{remark}
\begin{remark}
Let us analyze local (in time) convergence results as $\tau\downarrow 0$. Actually, the choice of initial data exerts some influence on the rate of convergence, which can be summarized in the next table (conv. means convergence):
\begin{table}[!htbp]
	\centering
\begin{tabular}{ccc}
	\toprule  
	energy terms & $v_2\neq \Delta u_0+\Delta u_1$& $v_2=\Delta u_0+\Delta u_1$\\
	\midrule  
	$\|w_{tt}(t,\cdot)\|_{L^2(\mb{R}^n)}^2+\|w_t(t,\cdot)\|_{L^2(\mb{R}^n)}^2$& no conv.& conv. with the rate $\ml{O}(\tau)$\\
	$\|\nabla w_{t}(t,\cdot)\|_{L^2(\mb{R}^n)}^2+\|\nabla w(t,\cdot)\|_{L^2(\mb{R}^n)}^2$& conv. with the rate $\ml{O}(\tau)$& conv. with the rate $\ml{O}(\tau^2)$\\
	\bottomrule 
\end{tabular}
\caption{Convergence analysis}\label{tab_1}
\end{table}
\end{remark}
\begin{remark}
Let us turn to analyze global (in time) convergence results as $\tau\downarrow0$. Indeed, the choice of initial data plays an extremely important role on the results. Concerning the general case $|P_{u_1}|\neq0$, to arrive at the convergence result with respect to $\tau\downarrow0$ for all $t>0$, we should consider higher dimensional cases $n\geqslant 3$ in Theorem \ref{Thm_Sing_Lim_Energy}. Pay attention, since $|P_{u_1}|\neq0$, it holds $u_1\neq0$. Furthermore, the rates of convergence are exactly the same as those in Table \ref{tab_1}.
\end{remark}
\begin{proof}
To begin with the proof, let us define an energy
\begin{align*}
\ml{E}_1[w](t)&:=\tau\|w_{tt}(t,\cdot)\|_{L^2(\mb{R}^n)}^2+\|\nabla w_t(t,\cdot)\|_{L^2(\mb{R}^n)}^2-2\int_{\mb{R}^n}\Delta w(t,x)w_t(t,x)\mathrm{d}x\\
&\quad\ +\gamma\int_0^tg(t-s)\|\nabla w(t,\cdot)-\nabla w(s,\cdot)\|_{L^2(\mb{R}^n)}^2\mathrm{d}s+g(t)\|\nabla w(t,\cdot)\|_{L^2(\mb{R}^n)}^2\\
&\quad\ +2\int_0^tg(t-s)\int_{\mb{R}^n}\Delta w(s,x)w_t(t,x)\mathrm{d}x\mathrm{d}s.
\end{align*}
According to the equation in \eqref{Eq_Sing_Lim_Energy}, one may get
\begin{align*}
2\tau\int_{\mb{R}^n}w_{ttt}(t,x)w_{tt}(t,x)\mathrm{d}x&=2\int_{\mb{R}^n}(-F(u)-w_{tt}+\Delta w+\Delta w_t-g\ast\Delta w)(t,x)w_{tt}(t,x)\mathrm{d}x\\
&=-2\int_{\mb{R}^n}F(u)(t,x)w_{tt}(t,x)\mathrm{d}x-2\|w_{tt}(t,\cdot)\|_{L^2(\mb{R}^n)}^2\\
&\quad+2\int_{\mb{R}^n}\Delta w(t,x)w_{tt}(t,x)\mathrm{d}x-2\int_{\mb{R}^n}\nabla w_{tt}(t,x)\cdot\nabla w_t(t,x)\mathrm{d}x\\
&\quad-2\int_{\mb{R}^n}(g\ast\Delta w)(t,x)w_{tt}(t,x)\mathrm{d}x,
\end{align*}
where we used integrations by parts in the last line.\\
Hence, we arrive at
\begin{align}\label{d_E1}
\frac{\mathrm{d}}{\mathrm{d}t}\ml{E}_1[w](t)&=-2\int_{\mb{R}^n}F(u)(t,x)w_{tt}(t,x)\mathrm{d}x-2\|w_{tt}(t,\cdot)\|_{L^2(\mb{R}^n)}^2+2\|\nabla w_t(t,\cdot)\|_{L^2(\mb{R}^n)}^2\notag\\
&\quad-\gamma^2\int_0^tg(t-s)\|\nabla w(t,\cdot)-\nabla w(s,\cdot)\|_{L^2(\mb{R}^n)}^2\mathrm{d}s-\gamma g(t)\|\nabla w(t,\cdot)\|_{L^2(\mb{R}^n)}^2,
\end{align}
by applying the fact that
\begin{align*}
&2\gamma\int_0^tg(t-s)\int_{\mb{R}^n}\nabla w_t(t,x)\cdot\nabla w(t,x)\mathrm{d}x\mathrm{d}s\\
&=-2\int_{\mb{R}^n}\Delta w(t,x)w_t(t,x)\mathrm{d}x-2g(t)\int_{\mb{R}^n}\nabla w_t(t,x)\cdot\nabla w(t,x)\mathrm{d}x.
\end{align*}
We next define another energy
\begin{align*}
\ml{E}_2[w](t)&:=\|\nabla w(t,\cdot)\|_{L^2(\mb{R}^n)}^2+\|w_t(t,\cdot)\|_{L^2(\mb{R}^n)}^2+2\tau\int_{\mb{R}^n}w_{tt}(t,x)w_t(t,x)\mathrm{d}x\\
&\quad\ +\int_0^tg(t-s)\|\nabla w(t,\cdot)-\nabla w(s,\cdot)\|_{L^2(\mb{R}^n)}^2\mathrm{d}s+\frac{1}{\gamma}(g(t)-1)\|\nabla w(t,\cdot)\|_{L^2(\mb{R}^n)}^2.
\end{align*}
Differentiating the above energy with respect to the time variable, it yields
\begin{align}\label{d_E2}
\frac{\mathrm{d}}{\mathrm{d}t}\ml{E}_2[w](t)&=-2\int_{\mb{R}^n}F(u)(t,x)w_t(t,x)\mathrm{d}x+2\tau\|w_{tt}(t,\cdot)\|_{L^2(\mb{R}^n)}^2-2\|\nabla w_t(t,\cdot)\|_{L^2(\mb{R}^n)}^2\notag\\
&\quad-\gamma\int_0^tg(t-s)\|\nabla w(t,\cdot)-\nabla w(s,\cdot)\|_{L^2(\mb{R}^n)}^2\mathrm{d}s-g(t)\|\nabla w(t,\cdot)\|_{L^2(\mb{R}^n)}^2.
\end{align}
Let us now introduce a total energy such that
\begin{align*}
\ml{E}_{\mathrm{T}}[w](t):=\ml{E}_1[w](t)+k\ml{E}_2[w](t),
\end{align*}
with a suitable positive constant $k$ (independent of $\tau$) to be determined later, which can be represented by
\begin{align*}
\ml{E}_{\mathrm{T}}[w](t)&=\|\nabla w_t(t,\cdot)+\nabla w(t,\cdot)\|_{L^2(\mb{R}^n)}^2+\tau\|w_{tt}(t,\cdot)+kw_t(t,\cdot)\|_{L^2(\mb{R}^n)}^2+k(1-\tau k)\|w_t(t,\cdot)\|_{L^2(\mb{R}^n)}^2\\
&\quad+(k-1)\|\nabla w(t,\cdot)\|_{L^2(\mb{R}^n)}^2+(\gamma+k)\int_0^tg(t-s)\|\nabla w(t,\cdot)-\nabla w(s,\cdot)\|_{L^2(\mb{R}^n)}^2\mathrm{d}s\\
&\quad+\left(g(t)+\frac{k}{\gamma}(g(t)-1)\right)\|\nabla w(t,\cdot)\|_{L^2(\mb{R}^n)}^2+2\int_{\mb{R}^n}(g\ast(w_t\Delta w))(t,x)\mathrm{d}x.
\end{align*}
In other words, from \eqref{d_E1} and \eqref{d_E2} we found
\begin{align}\label{Est_-1}
&\frac{\mathrm{d}}{\mathrm{d}t}\ml{E}_{\mathrm{T}}[w](t)+(2-2\tau k)\|w_{tt}(t,\cdot)\|_{L^2(\mb{R}^n)}^2+(2k-2)\|\nabla w_t(t,\cdot)\|_{L^2(\mb{R}^n)}^2\notag\\
&+(\gamma+k)g(t)\|\nabla w(t,\cdot)\|_{L^2(\mb{R}^n)}^2+(\gamma^2+\gamma k)\int_0^tg(t-s)\|\nabla w(t,\cdot)-\nabla w(s,\cdot)\|_{L^2(\mb{R}^n)}^2\mathrm{d}s\notag\\
&=-2\int_{\mb{R}^n}F(u)(t,x)(w_{tt}(t,x)+kw_t(t,x))\mathrm{d}x.
\end{align}
From the equation in \eqref{VDW_Equation}, we may rewrite
\begin{align*}
F(u)(t,x)&=\tau\left(-\Delta u+\Delta u_t+\Delta^2 u+\Delta^2 u_t-g\ast\Delta^2 u+\gamma g\ast\Delta u\right)(t,x)\\
&=\mathrm{div}\left(\tau\nabla\left(-u+u_t+\Delta u+\Delta u_t-g\ast\Delta u+\gamma g\ast u\right)\right)(t,x).
\end{align*}
By using integration by parts and the Cauchy-Schwarz inequality, we get
\begin{align}\label{Est_00}
&-2\int_{\mb{R}^n}F(u)(t,x)(w_{tt}(t,x)+kw_t(t,x))\mathrm{d}x\notag\\
&\leqslant\frac{\tau^2}{4\varepsilon_1}\left(\sum\limits_{j,k=0,1}\left\|\partial_t^j\Delta^{k+1}u(t,\cdot)\right\|_{L^2(\mb{R}^n)}^2+\gamma^2\|(g\ast\Delta u)(t,\cdot)\|_{L^2(\mb{R}^n)}^2+\left\|(g\ast\Delta^2u)(t,\cdot)\right\|_{L^2(\mb{R}^n)}^2\right)\notag\\
&\quad+\frac{k^2\tau^2}{4\varepsilon_1}\left(\sum\limits_{j,k=0,1}\left\|\partial_t^j\nabla\Delta^k u(t,\cdot)\right\|_{L^2(\mb{R}^n)}^2+\gamma^2\|(g\ast\nabla u)(t,\cdot)\|_{L^2(\mb{R}^n)}^2+\|(g\ast\nabla\Delta u)(t,\cdot)\|_{L^2(\mb{R}^n)}^2\right)\notag\\
&\quad+6\varepsilon_1\left(\|w_{tt}(t,\cdot)\|_{L^2(\mb{R}^n)}^2+\|\nabla w_t(t,\cdot)\|_{L^2(\mb{R}^n)}^2\right),
\end{align}
where $\varepsilon_1$ is a suitable constant (independent of $\tau$) to be fixed later.\\
For one thing, let us recall some estimates derived in Theorem \ref{Thm_Estimates} that
\begin{align}\label{Est_01}
&\sum\limits_{j,k=0,1}\left(\left\|\partial_t^j\Delta^{k+1}u(t,\cdot)\right\|_{L^2(\mb{R}^n)}^2+\left\|\partial_t^j\nabla \Delta^{k}u(t,\cdot)\right\|_{L^2(\mb{R}^n)}^2\right)\notag\\
&\leqslant C (1+t)^{-1-\frac{n}{2}}\|u_0\|^2_{H^4(\mb{R}^n)\cap L^{1,1}(\mb{R}^n)}+C(1+t)^{-\frac{n}{2}}\|u_1\|^2_{H^4(\mb{R}^n)\cap L^{1,1}(\mb{R}^n)},
\end{align}
where $C$ is a positive constant independent of $\tau$.\\
For another, since Minkowski's integral inequality and Theorem \ref{Thm_Estimates}, we see
\begin{align*}
\left\|(g\ast\nabla^k u)(t,\cdot)\right\|_{L^2(\mb{R}^n)}^2&=\int_{\mb{R}^n}\left|\int_0^tg(t-s)\nabla^k u(s,x)\mathrm{d}s\right|^2\mathrm{d}x\\
&\leqslant\left(\int_0^tg(t-s)\left\|\nabla^k u(s,\cdot)\right\|_{L^2(\mb{R}^n)}\mathrm{d}s\right)^2\\
&\leqslant C\left(\int_0^tg(t-s)(1+s)^{-\frac{k}{2}-\frac{n}{4}}\mathrm{d}s\right)^2\|u_0\|_{H^k(\mb{R}^n)\cap L^{1,1}(\mb{R}^n)}^2\\
&\quad+C\left(\int_0^tg(t-s)(1+s)^{-\frac{k-1}{2}-\frac{n}{4}}\mathrm{d}s\right)^2\|u_1\|_{H^{\max\{k-2,0\}}(\mb{R}^n)\cap L^{1,1}(\mb{R}^n)}^2
\end{align*}
for $k=1,2,3,4$. Again, $C$ is a positive constant independent of $\tau$. Next, we compute
\begin{align*}
\int_0^tg(t-s)(1+s)^{-\frac{k}{2}-\frac{n}{4}}\mathrm{d}s&=\mathrm{e}^{-\gamma t}\int_0^t\mathrm{e}^{\gamma s}(1+s)^{-\frac{k}{2}-\frac{n}{4}}\mathrm{d}s\\
&=\mathrm{e}^{-\gamma t}\int_0^{t/2}\mathrm{e}^{\gamma s}(1+s)^{-\frac{k}{2}-\frac{n}{4}}\mathrm{d}s+\mathrm{e}^{-\gamma t}\int_{t/2}^t\mathrm{e}^{\gamma s}(1+s)^{-\frac{k}{2}-\frac{n}{4}}\mathrm{d}s\\
&\leqslant \mathrm{e}^{-\gamma t/2}\int_0^{t/2}(1+s)^{-\frac{k}{2}-\frac{n}{4}}\mathrm{d}s+\mathrm{e}^{-\gamma t}(1+t/2)^{-\frac{k}{2}-\frac{n}{4}}\int_{t/2}^t\mathrm{e}^{\gamma s}\mathrm{d}s\\
&\leqslant\frac{4}{4-2k-n}\mathrm{e}^{-\gamma t/2}\left((1+t/2)^{1-\frac{k}{2}-\frac{n}{4}}-1\right)+\frac{1}{\gamma}(1+t/2)^{-\frac{k}{2}-\frac{n}{4}}\left(1-\mathrm{e}^{-\gamma t/2}\right)\\
&\leqslant C(1+t)^{-\frac{k}{2}-\frac{n}{4}}.
\end{align*}
Namely, the following estimates hold:
\begin{align}\label{Est_02}
\left\|(g\ast\nabla^k u)(t,\cdot)\right\|_{L^2(\mb{R}^n)}^2&\leqslant C(1+t)^{-k-\frac{n}{2}}\|u_0\|^2_{H^k(\mb{R}^n)\cap L^{1,1}(\mb{R}^n)}\notag\\
&\quad+C(1+t)^{1-k-\frac{n}{2}}\|u_1\|^2_{H^{\max\{k-2,0\}}(\mb{R}^n)\cap L^{1,1}(\mb{R}^n)}.
\end{align}
Combining with \eqref{Est_-1}, \eqref{Est_00}, \eqref{Est_01} and \eqref{Est_02}, they conclude
\begin{align}\label{Est_03}
&\frac{\mathrm{d}}{\mathrm{d}t}\ml{E}_{\mathrm{T}}[w](t)+(2-2\tau k-6\varepsilon_1)\|w_{tt}(t,\cdot)\|_{L^2(\mb{R}^n)}^2+(2k-2-6\varepsilon_1)\|\nabla w_t(t,\cdot)\|_{L^2(\mb{R}^n)}^2\notag\\
&+(\gamma+k)g(t)\|\nabla w(t,\cdot)\|_{L^2(\mb{R}^n)}^2+\gamma(\gamma+ k)\int_0^tg(t-s)\|\nabla w(t,\cdot)-\nabla w(s,\cdot)\|_{L^2(\mb{R}^n)}^2\mathrm{d}s\notag\\
&\leqslant  C\tau^2 (1+t)^{-1-\frac{n}{2}}\|u_0\|^2_{H^4(\mb{R}^n)\cap L^{1,1}(\mb{R}^n)}+C\tau^2(1+t)^{-\frac{n}{2}}\|u_1\|^2_{H^4(\mb{R}^n)\cap L^{1,1}(\mb{R}^n)}.
\end{align}
By choosing $k\in[1+3\varepsilon_1,(1-3\varepsilon_1)/\tau]$ carrying $\varepsilon_1\in(0,(1-\tau)/(3(1+\tau))]$, we claim that
\begin{align*}
2-2\tau k-6\varepsilon_1\geqslant 0\ \ \mbox{and}\ \ 2k-2-6\varepsilon_1\geqslant0.
\end{align*}
Here, letting $0<\tau\ll 1$, we are able to choose parameters $k$ and $\varepsilon_1$ independent of $\tau$. For example, one may choose $k=2$, $\varepsilon_1=1/6$ as $0<\tau\ll 1$.\\
Then, integrating \eqref{Est_03} over $[0,t]$, it shows immediately
\begin{align}\label{Eq_05}
&\ml{E}_{\mathrm{T}}[w](t)+C\int_0^t\left(\|w_{tt}(s,\cdot)\|_{L^2(\mb{R}^n)}^2+\|\nabla w_t(s,\cdot)\|_{L^2(\mb{R}^n)}^2+g(s)\|\nabla w(s,\cdot)\|_{L^2(\mb{R}^n)}^2\right)\notag\\
&+C\int_0^t\int_0^s g(s-\eta)\|\nabla w(s,\cdot)-\nabla w(\eta,\cdot)\|_{L^2(\mb{R}^n)}^2\mathrm{d}\eta\mathrm{d}s\notag\\
&\leqslant \tau\|v_2-\Delta u_0-\Delta u_1\|_{L^2(\mb{R}^n)}^2+C\tau^2\|u_0\|_{H^4(\mb{R}^n)\cap L^{1,1}(\mb{R}^n)}^2+C\tau^2\kappa_n(t)\|u_1\|^2_{H^4(\mb{R}^n)\cap L^{1,1}(\mb{R}^n)}
\end{align}
In the above calculations, we employed $\ml{E}_{\mathrm{T}}[w](0)=\tau\|w_2\|_{L^2(\mb{R}^n)}^2$.\\
Finally, motived by \cite{LasieckaWang2016}, next standard energy can be shown:
\begin{align*}
\ml{E}_{\mathrm{S}}[w](t)&:=\tau\|w_{tt}(t,\cdot)\|_{L^2(\mb{R}^n)}^2+\|\nabla w_t(t,\cdot)\|_{L^2(\mb{R}^n)}^2+\|\nabla w(t,\cdot)\|_{L^2(\mb{R}^n)}^2+\tau\|w_t(t,\cdot)\|_{L^2(\mb{R}^n)}^2\\
&\quad\ +\gamma\int_0^tg(t-s)\|\nabla w(t,\cdot)-\nabla w(s,\cdot)\|_{L^2(\mb{R}^n)}^2\mathrm{d}s.
\end{align*}
Performing the similar computation as the proof of Lemma 3.1 in \cite{LasieckaWang2016}, there is a positive constant $C$ independent of $\tau$ such that 
\begin{align}\label{Eq_06}
C\ml{E}_{\mathrm{S}}[w](t)\leqslant \ml{E}_{\mathrm{T}}[w](t).
\end{align}
Let us explain more detail for the terms including $\tau$ in the total energy $\ml{E}_{\mathrm{T}}[w](t)$. We apply Lemma 2.6 in \cite{LasieckaWang2016} to get
\begin{align*}
&\tau\|w_{tt}(t,\cdot)+kw_t(t,\cdot)\|_{L^2(\mb{R}^n)}^2+k(1-\tau k)\|w_t(t,\cdot)\|_{L^2(\mb{R}^n)}^2\\
&=\tau\left(\|w_{tt}(t,\cdot)+kw_t(t,\cdot)\|_{L^2(\mb{R}^n)}^2+\frac{1-\tau k}{\tau k}\|kw_t(t,\cdot)\|_{L^2(\mb{R}^n)}^2\right)\\
&\geqslant C_1\tau \left(\|w_{tt}(t,\cdot)\|_{L^2(\mb{R}^n)}^2+\|w_t(t,\cdot)\|_{L^2(\mb{R}^n)}^2\right),
\end{align*}
where
\begin{align*}
C_1=\min\left\{1-\frac{2}{2+C_0},\frac{C_0}{2}\right\}\ \ \mbox{and}\ \ C_0:=\frac{1-\tau k}{\tau k}.
\end{align*}
Here, we observe that
\begin{align*}
 C_1\tau=\min\left\{\tau-\frac{2k\tau^2}{1+\tau k},\frac{1}{2k}-\frac{\tau}{2} \right\}\geqslant C\tau\ \ \mbox{as}\ \ 0<\tau\ll 1.
\end{align*}
In conclusion, summarizing \eqref{Eq_05} and \eqref{Eq_06}, our proof is complete.
\end{proof}
\subsection{Singular limit relation on the solution itself}
In the last subsection, we found a convergence result for standard energy, in particular, the solution of the MGT equation with memory converges to that of the viscoelastic damped wave equation with memory as $\tau\downarrow 0$ in the $\dot{H}^1$ space. The natural question is the convergence behavior in the $L^2$ space, which implies the singular limit for the solution itself.

 Nevertheless, the singular limit for the solution itself is not a simple generalization of the previous theorem because the $L^2$ norm for the solution itself always does not be included in classical or standard energies, especially, for the massless models. With the aim of overcoming the difficulty, we are motived by the papers \cite{Ike-2003,Ikehata-Nishihara-2003} to define an integral form such that
 \begin{align}\label{Variable_New}
 W(t,x):=\int_0^tw(s,x)\mathrm{d}s.
 \end{align}
 
 In order to compensate the influence of memory term in \eqref{VDW_Equation} as well as \eqref{Sing_Lim_MGT}, we will transfer the second- and third-order integro-differential equations to fourth-order differential equations. Due to the structure of the relaxation function \eqref{Exponential_Decay_Relax}, the equation in \eqref{VDW_Equation} can be rewritten by acting the operator $\partial_t+\gamma\ml{I}$ as follows:
 \begin{align}\label{Eq_07}
 u_{ttt}-\Delta u_{tt}+\gamma u_{tt}-(\gamma+1)\Delta u_t+(1-\gamma)\Delta u=0.
 \end{align}
 By processing the same way, the equation in \eqref{Sing_Lim_MGT} can be rewritten by
 \begin{align}\label{Eq_08}
 \tau v_{\tau,tttt}+(1+\gamma \tau)v_{\tau,ttt}-\Delta v_{\tau,tt}+\gamma v_{\tau,tt}-(\gamma+1)\Delta v_{\tau,t}+(1-\gamma)\Delta v_{\tau}=0.
 \end{align}
 Next, with the computations that $\tau\partial_t$ acts on \eqref{Eq_07} and the resultant equation pluses \eqref{Eq_07}, it yields
 \begin{align}\label{Eq_09}
\tau u_{tttt}+(1+\gamma\tau) u_{ttt}-\Delta u_{tt}+\gamma u_{tt}-(\gamma+1)\Delta u_t+(1-\gamma)\Delta u=\widetilde{F}(u),
 \end{align} 
 where $\widetilde{F}(u)=\widetilde{F}(u)(t,x)$ such that
 \begin{align*}
 \widetilde{F}(u)(t,x):=\tau\left(\Delta u_{ttt}+(\gamma+1)\Delta u_{tt}-(1-\gamma)\Delta u_t\right)(t,x).
 \end{align*}
 From \eqref{Eq_08} and \eqref{Eq_09}, the unknown $w=w(t,x)$ according to the definition \eqref{Diff} solves the following Cauchy problem for fourth-order (in time) evolution partial differential equation:
 \begin{align}\label{Eq_10}
 \begin{cases}
 \tau w_{tttt}+(1+\gamma\tau) w_{ttt}-\Delta w_{tt}+\gamma w_{tt}-(\gamma+1)\Delta w_t+(1-\gamma)\Delta w=-\widetilde{F}(u),&x\in\mb{R}^n,\ t>0,\\
 (w,w_t,w_{tt},w_{ttt})(0,x)=(0,0,w_2,w_3)(x),&x\in\mb{R}^n,
 \end{cases}
 \end{align}
 where initial data is given by
 \begin{align*}
 w_2(x):&=v_2(x)-(\Delta u_0(x)+\Delta u_1(x)),\\
 w_3(x):&=-\frac{1}{\tau}v_2(x)-\Delta^2u_0(x)-\Delta^2 u_1(x)+\frac{1-\tau}{\tau}\Delta u_1(x)+\frac{1+\tau}{\tau}\Delta u_0(x).
 \end{align*}
 From the new variable \eqref{Variable_New} and the equation in \eqref{Eq_10}, it shows that
\begin{align}\label{Eq_11}
&\tau W_{tttt}+(1+\gamma\tau) W_{ttt}-\Delta W_{tt}+\gamma W_{tt}-(\gamma+1)\Delta W_t+(1-\gamma)\Delta W\notag\\
&=\tau w_{ttt}+(1+\gamma\tau)w_{tt}-\Delta w_t+\gamma w_t-(\gamma+1)\Delta w+(1-\gamma)\int_0^t\Delta w(s,x)\mathrm{d}s\notag\\
&=-\int_0^t\widetilde{F}(u)(s,x)\mathrm{d}s-\tau\ml{I}_{\mathrm{Data}}(x),
\end{align}
where the combination of initial data is independent of $\tau$ fulfilling
\begin{align*}
\ml{I}_{\mathrm{Data}}(x):=-\gamma v_2(x)+\Delta^2u_0(x)+\Delta^2u_1(x)+(1+\gamma)\Delta u_1(x)-(1-\gamma)\Delta u_0(x).
\end{align*}
For this reason, we will mainly study the equation \eqref{Eq_11} since the fact that
\begin{align}\label{Rela_Sing_Lim_Solu}
\|W_t(t,\cdot)\|_{L^2(\mb{R}^n)}^2=\|w(t,\cdot)\|_{L^2(\mb{R}^n)}^2.
\end{align}

\begin{theorem}\label{Thm_Sing_Lim_Solution}
Let $\gamma>5$ and $n\geqslant 3$. Let us assume $(u_0,u_1)\in \widetilde{\ml{D}}_4(\mb{R}^n)$ and $v_2\in L^2(\mb{R}^n)$ carrying $|x|v_2\in L^2(\mb{R}^n)$, where $u_0$ and $u_1$ are not zero simultaneously. Then, the difference $w=w(t,x)$ defined in \eqref{Diff} fulfills the following estimate for $0<\tau\ll 1$: 
\begin{align*}
\|w(t,\cdot)\|_{L^2(\mb{R}^n)}^2&\leqslant C\tau^2\left(\|v_2\|_{L^2(\mb{R}^n)}^2+\|\,|x|v_2\|_{L^2(\mb{R}^n)}^2\right)+C\tau\|v_2-\Delta u_0-\Delta u_1\|_{L^2(\mb{R}^n)}^2\\
&\quad+ C\tau^2 \left(\|u_0\|_{H^4(\mb{R}^n)\cap L^{1,1}(\mb{R}^n)}^2+\|u_1\|_{H^4(\mb{R}^n)\cap L^{1,1}(\mb{R}^n)}^2\right).
\end{align*}
\end{theorem}
\begin{remark}
Under the considerations of additional weighted $L^1$ regularity on $u_0, u_1$ and $n\geqslant 3$, the global (in time) convergence result can be observed for the solution itself. Precisely, concerning any $t>0$, the following results hold:
\begin{itemize}
	\item if $v_2\neq\Delta u_0+\Delta u_1$, then the rate of convergence is $\ml{O}(\tau)$ as $\tau\downarrow0$;
	\item if $v_2=\Delta u_0+\Delta u_1 $, then the rate of convergence is improved by $\ml{O}(\tau^2)$ as $\tau\downarrow0$.
\end{itemize}
\end{remark}
\begin{proof}
To derive our desired result, we need to construct several different energies for the fourth-order (in time) equation, which are quite different from the proof of Theorem \ref{Thm_Sing_Lim_Energy}. We will show this procedure by several steps.\\

\noindent\textbf{Step 1. Construction of the first energy by multiplying $W_{ttt}$}\\
Let us first multiply \eqref{Eq_11} by $W_{ttt}$ and integrate the resultant over $\mb{R}^n$. Then, we may arrive at
\begin{align*}
\frac{\mathrm{d}}{\mathrm{d}t}\ml{E}_{m1}[W](t)+\ml{R}_{m1}[W](t)=-\int_{\mb{R}^n}\int_0^t\widetilde{F}(u)(s,x)\mathrm{d}s\,W_{ttt}(t,x)\mathrm{d}x-\tau\int_{\mb{R}^n}\ml{I}_{\mathrm{Data}}(x)W_{ttt}(t,x)\mathrm{d}x,
\end{align*}
where
\begin{align*}
\ml{E}_{m1}[W](t)&:=\frac{\tau}{2}\|W_{ttt}(t,\cdot)\|_{L^2(\mb{R}^n)}^2+\frac{1}{2}\|\nabla W_{tt}(t,\cdot)\|_{L^2(\mb{R}^n)}^2+\frac{\gamma}{2}\|W_{tt}(t,\cdot)\|_{L^2(\mb{R}^n)}^2\\
&\quad+\frac{1-\gamma}{2}\|\nabla W_t(t,\cdot)\|_{L^2(\mb{R}^n)}^2+(\gamma+1)\int_{\mb{R}^n}\nabla W_t(t,x)\cdot\nabla W_{tt}(t,x)\mathrm{d}x\\
&\quad-(1-\gamma)\int_{\mb{R}^n}\nabla W(t,x)\cdot \nabla W_{tt}(t,x)\mathrm{d}x,
\end{align*}
and
\begin{align*}
\ml{R}_{m1}[W](t):=(1+\gamma\tau)\|W_{ttt}(t,\cdot)\|_{L^2(\mb{R}^n)}^2-(\gamma+1)\|\nabla W_{tt}(t,\cdot)\|_{L^2(\mb{R}^n)}^2.
\end{align*}

\noindent\textbf{Step 2. Construction of the second energy by multiplying $W_{tt}$}\\
Multiplying \eqref{Eq_11} by $W_{tt}$ and integrating it over $\mb{R}^n$, we may get
\begin{align*}
\frac{\mathrm{d}}{\mathrm{d}t}\ml{E}_{m2}[W](t)+\ml{R}_{m2}[W](t)=-\int_{\mb{R}^n}\int_0^t\widetilde{F}(u)(s,x)\mathrm{d}s\,W_{tt}(t,x)\mathrm{d}x-\tau\int_{\mb{R}^n}\ml{I}_{\mathrm{Data}}(x)W_{tt}(t,x)\mathrm{d}x,
\end{align*}
where
\begin{align*}
\ml{E}_{m2}[W](t)&:=\tau\int_{\mb{R}^n}W_{ttt}(t,x)W_{tt}(t,x)\mathrm{d}x+\frac{1+\gamma\tau}{2}\|W_{tt}(t,\cdot)\|_{L^2(\mb{R}^n)}^2+\frac{\gamma+1}{2}\|\nabla W_t(t,\cdot)\|_{L^2(\mb{R}^n)}^2\\
&\quad-(1-\gamma)\int_{\mb{R}^n}\nabla W(t,x)\cdot\nabla W_t(t,x)\mathrm{d}x,
\end{align*}
and
\begin{align*}
\ml{R}_{m2}[W](t)&:=-\tau\|W_{ttt}(t,\cdot)\|_{L^2(\mb{R}^n)}^2+\|\nabla W_{tt}(t,\cdot)\|_{L^2(\mb{R}^n)}^2+\gamma\|W_{tt}(t,\cdot)\|_{L^2(\mb{R}^n)}^2\\
&\quad+(1-\gamma)\|\nabla W_t(t,\cdot)\|_{L^2(\mb{R}^n)}^2.
\end{align*}

\noindent\textbf{Step 3. Construction of the third energy by multiplying $W_{t}$}\\
Similarly, by multiplying \eqref{Eq_11} by $W_{t}$ and integrating the equality over $\mb{R}^n$, one finds
\begin{align*}
\frac{\mathrm{d}}{\mathrm{d}t}\ml{E}_{m3}[W](t)+\ml{R}_{m3}[W](t)=-\int_{\mb{R}^n}\int_0^t\widetilde{F}(u)(s,x)\mathrm{d}s\,W_{t}(t,x)\mathrm{d}x-\tau\int_{\mb{R}^n}\ml{I}_{\mathrm{Data}}(x)W_{t}(t,x)\mathrm{d}x,
\end{align*}
where
\begin{align*}
\ml{E}_{m3}[W](t)&:=\tau\int_{\mb{R}^n}W_{ttt}(t,x)W_{t}(t,x)\mathrm{d}x+(1+\gamma\tau)\int_{\mb{R}^n}W_{tt}(t,x)W_t(t,x)\mathrm{d}x+\frac{1}{2}\|\nabla W_t(t,\cdot)\|_{L^2(\mb{R}^n)}^2\\
&\quad-\frac{\tau}{2}\|W_{tt}(t,\cdot)\|_{L^2(\mb{R}^n)}^2+\frac{\gamma}{2}\|W_t(t,\cdot)\|_{L^2(\mb{R}^n)}^2+\frac{\gamma-1}{2}\|\nabla W(t,\cdot)\|_{L^2(\mb{R}^n)}^2,
\end{align*}
and
\begin{align*}
\ml{R}_{m3}[W](t)&:=-(1+\gamma\tau)\|W_{tt}(t,\cdot)\|_{L^2(\mb{R}^n)}^2+(\gamma+1)\|\nabla W_t(t,\cdot)\|_{L^2(\mb{R}^n)}^2.
\end{align*}

\noindent\textbf{Step 4. Construction of the total energy}\\
Let us now define
\begin{align*}
\ml{E}_m[W](t)&:=\ml{E}_{m1}[W](t)+k_1\ml{E}_{m2}[W](t)+k_2\ml{E}_{m3}[W](t),\\
\ml{R}_m[W](t)&:=\ml{R}_{m1}[W](t)+k_1\ml{R}_{m2}[W](t)+k_2\ml{R}_{m3}[W](t),
\end{align*}
which lead to
\begin{align}\label{Total_Est}
\frac{\mathrm{d}}{\mathrm{d}t}\ml{E}_m[W](t)+\ml{R}_m[W](t)=\ml{F}_m[W](t),
\end{align}
where $k_1$, $k_2$ are suitably positive constants (independent of $\tau$) which will be determined later, and 
\begin{align*}
\ml{F}_m[W](t)&:=-\int_{\mb{R}^n}\int_0^t\widetilde{F}(u)(s,x)\mathrm{d}s\,(k_2W_t(t,x)+k_1W_{tt}(t,x)+W_{ttt}(t,x))\mathrm{d}x\\
&\quad-\tau\int_{\mb{R}^n}\ml{I}_{\mathrm{Data}}(x)(k_2W_t(t,x)+k_1W_{tt}(t,x)+W_{ttt}(t,x))\mathrm{d}x.
\end{align*}

\noindent\textbf{Step 5. Estimates of the right-hand side}\\
Considering
\begin{align*}
\int_0^t\widetilde{F}(u)(s,x)\mathrm{d}s&=\tau\left(\Delta u_{tt}+(\gamma+1)\Delta u_t-(1-\gamma)\Delta u\right)(t,x)+\tau\widetilde{I}_{\mathrm{Data}}(x),\\
\widetilde{I}_{\mathrm{Data}}(x)&=(-\Delta^2u_1-\Delta^2u_0-(\gamma+1)\Delta u_1+(1-\gamma)\Delta u_0)(x),
\end{align*}
there exists a positive constant $\varepsilon_2$ independent of $\tau$ such that
\begin{align*}
\ml{F}_m[W](t)&\leqslant\frac{C\tau^2 }{4\varepsilon_2}\left(\|\nabla u_{tt}(t,\cdot)\|_{L^2(\mb{R}^n)}^2+\|\nabla u_t(t,\cdot)\|_{L^2(\mb{R}^n)}^2+\|\nabla u(t,\cdot)\|_{L^2(\mb{R}^n)}^2\right)\\
&\quad+\frac{C\tau^2}{2\varepsilon_2}\left(\|\Delta u_{tt}(t,\cdot)\|_{L^2(\mb{R}^n)}^2+\|\Delta u_t(t,\cdot)\|_{L^2(\mb{R}^n)}^2+\|\Delta u(t,\cdot)\|_{L^2(\mb{R}^n)}^2\right)\\
&\quad+3\varepsilon_2\left(\|\nabla W_t(t,\cdot)\|_{L^2(\mb{R}^n)}^2+\|W_{tt}(t,\cdot)\|_{L^2(\mb{R}^n)}^2+\|W_{ttt}(t,\cdot)\|_{L^2(\mb{R}^n)}^2\right)\\
&\quad+\tau \gamma \int_{\mb{R}^n}v_2(x)(k_2W_t(t,x)+k_1W_{tt}(t,x)+W_{ttt}(t,x))\mathrm{d}x.
\end{align*} 
where we observed that $\ml{I}_{\mathrm{Data}}(x)+\widetilde{\ml{I}}_{\mathrm{Data}}(x)=-\gamma v_2(x)$.\\
Recalling the result in Theorem \ref{Thm_Estimates} or viewing of \eqref{Est_01} and \eqref{Est_02}, we get
\begin{align}\label{Eq_13}
\ml{F}_m[W](t)&\leqslant \frac{C\tau^2}{\varepsilon_2}\left((1+t)^{-1-\frac{n}{2}}\|u_0\|_{H^4(\mb{R}^n)\cap L^{1,1}(\mb{R}^n)}^2+(1+t)^{-\frac{n}{2}}\|u_1\|_{H^4(\mb{R}^n)\cap L^{1,1}(\mb{R}^n)}^2\right)\notag\\
&\quad +3\varepsilon_2\left(\|\nabla W_t(t,\cdot)\|_{L^2(\mb{R}^n)}^2+\|W_{tt}(t,\cdot)\|_{L^2(\mb{R}^n)}^2+\|W_{ttt}(t,\cdot)\|_{L^2(\mb{R}^n)}^2\right)\notag\\
&\quad+\frac{\mathrm{d}}{\mathrm{d}t}\left(\tau\gamma \int_{\mb{R}^n}v_2(x)(k_2W(t,x)+k_1W_t(t,x)+W_{tt}(t,x))\mathrm{d}x\right).
\end{align}
\noindent\textbf{Step 6. Estimates of the energy}\\
Let us begin with the estimate of remainder terms associated with the right-hand side. We get
\begin{align}\label{Eq_12}
&\ml{R}_m[W](t)-3\varepsilon_2\left(\|\nabla W_t(t,\cdot)\|_{L^2(\mb{R}^n)}^2+\|W_{tt}(t,\cdot)\|_{L^2(\mb{R}^n)}^2+\|W_{ttt}(t,\cdot)\|_{L^2(\mb{R}^n)}^2\right)\notag\\
&=(1+\gamma\tau-k_1\tau-3\varepsilon_2)\|W_{ttt}(t,\cdot)\|_{L^2(\mb{R}^n)}^2+(\gamma k_1-(1+\gamma \tau)k_2-3\varepsilon_2)\|W_{tt}(t,\cdot)\|_{L^2(\mb{R}^n)}^2\notag\\
&\quad+(k_1(1-\gamma)+k_2(\gamma-1)-3\varepsilon_2)\|\nabla W_t(t,\cdot)\|_{L^2(\mb{R}^n)}^2+(k_1-(\gamma+1))\|\nabla W_{tt}(t,\cdot)\|_{L^2(\mb{R}^n)}^2.
\end{align}
To guarantee the nonnegativity of all coefficients on the right-hand sides of the previous identity, one should treat
\begin{align}\label{Parameter_Range_01}
\begin{cases}
1+\gamma\tau-k_1\tau-3\varepsilon_2\geqslant 0,\\
\gamma k_1-(1+\gamma \tau)k_2-3\varepsilon_2\geqslant 0,\\
k_1(1-\gamma)+k_2(\gamma-1)-3\varepsilon_2\geqslant 0,\\
k_1-(\gamma+1)\geqslant 0.
\end{cases}
\end{align}
Indeed, let us choose $\varepsilon_2$ to be a sufficient small and fixed positive constant, $0<\tau\ll 1$ and $k_1=\gamma+1$. Then, it holds $k_2\in(1+\gamma,\gamma+\gamma^2)$ for all $\gamma>1$. Simultaneously, we assert that $\varepsilon_2,k_1,k_2$ are independent of $\tau$.

In other words, from \eqref{Total_Est}, \eqref{Eq_13} and \eqref{Eq_12}, we have
\begin{align*}
&\frac{\mathrm{d}}{\mathrm{d}t}\left(\ml{E}_m[W](t)-\tau\gamma \int_{\mb{R}^n}v_2(x)(k_2W(t,x)+k_1W_t(t,x)+W_{tt}(t,x))\mathrm{d}x\right)\\
&\leqslant C\tau^2 \left((1+t)^{-1-\frac{n}{2}}\|u_0\|_{H^4(\mb{R}^n)\cap L^{1,1}(\mb{R}^n)}^2+(1+t)^{-\frac{n}{2}}\|u_1\|_{H^4(\mb{R}^n)\cap L^{1,1}(\mb{R}^n)}^2\right).
\end{align*}
Integrating the last inequality over $[0,t]$, it yields
\begin{align*}
\ml{E}_m[W](t)&\leqslant\tau\gamma \int_{\mb{R}^n}v_2(x)(k_2W(t,x)+k_1W_t(t,x)+W_{tt}(t,x))\mathrm{d}x+\frac{\tau}{2}\|v_2-\Delta u_0-\Delta u_1\|_{L^2(\mb{R}^n)}^2\\
&\quad+ C\tau^2 \left(\|u_0\|_{H^4(\mb{R}^n)\cap L^{1,1}(\mb{R}^n)}^2+\kappa_n(t)\|u_1\|_{H^4(\mb{R}^n)\cap L^{1,1}(\mb{R}^n)}^2\right),
\end{align*}
where $\kappa_n(t)$ was defined in \eqref{kappa}.\\
Employing Hardy's inequality for $n\geqslant 3$, it holds
\begin{align*}
\tau k_2\gamma\int_{\mb{R}^n}v_2(x)W(t,x)\mathrm{d}x&\leqslant\frac{n\gamma^2k_2^2\tau^2}{4(n-2)\varepsilon_3}\|\,|x|v_2\|_{L^2(\mb{R}^n)}^2+\frac{(n-2)\varepsilon_3}{n}\int_{\mb{R}^n}\frac{|W(t,x)|^2}{|x|^2}\mathrm{d}x\\
&\leqslant\frac{C\tau^2}{\varepsilon_3}\|\,|x|v_2\|_{L^2(\mb{R}^n)}^2+\varepsilon_3\|\nabla W(t,\cdot)\|_{L^2(\mb{R}^n)}^2,
\end{align*}
where $\varepsilon_3$ is a suitably positive constant independent of $\tau$.\\
Let us use Cauchy-Schwarz inequality with $\varepsilon_3$ to arrive at
\begin{align*}
\tau\gamma\int_{\mb{R}^n}v_2(x)(k_1W_t(t,x)+W_{tt}(t,x))\mathrm{d}x\leqslant\frac{C\tau^2}{\varepsilon_3}\|v_2\|_{L^2(\mb{R}^n)}^2+\varepsilon_3\left(\|W_t(t,\cdot)\|_{L^2(\mb{R}^n)}^2+\|W_{tt}(t,\cdot)\|_{L^2(\mb{R}^n)}^2\right).
\end{align*}
We have derived
\begin{align*}
\widetilde{\ml{E}}_m[W](t)&=\ml{E}_m[W](t)-\varepsilon_3\left(\|\nabla W(t,\cdot)\|_{L^2(\mb{R}^n)}^2+\|W_t(t,\cdot)\|_{L^2(\mb{R}^n)}^2+\|W_{tt}(t,\cdot)\|_{L^2(\mb{R}^n)}^2\right)\\
&\leqslant\frac{C\tau^2}{\varepsilon_3}\left(\|v_2\|_{L^2(\mb{R}^n)}^2+\|\,|x|v_2\|_{L^2(\mb{R}^n)}^2\right)+\frac{\tau}{2}\|v_2-\Delta u_0-\Delta u_1\|_{L^2(\mb{R}^n)}^2\\
&\quad+ C\tau^2 \left(\|u_0\|_{H^4(\mb{R}^n)\cap L^{1,1}(\mb{R}^n)}^2+\|u_1\|_{H^4(\mb{R}^n)\cap L^{1,1}(\mb{R}^n)}^2\right),
\end{align*}
Eventually, we only need to give a lower bound estimate for the modified energy $\widetilde{\ml{E}}_m[W](t)$. By applying Cauchy-Schwarz inequality again associated with positive constants $\varepsilon_4,\dots,\varepsilon_9$, one may deduce
\begin{align*}
&(\gamma+1)\int_{\mb{R}^n}\nabla W_t(t,x)\cdot\nabla W_{tt}(t,x)\mathrm{d}x-(1-\gamma)\int_{\mathbb{R}^n}\nabla W(t,x)\cdot\nabla W_{tt}(t,x)\mathrm{d}x\\
&+\tau k_1\int_{\mb{R}^n}W_{ttt}(t,x)W_{tt}(t,x)\mathrm{d}x-(1-\gamma)k_1\int_{\mb{R}^n}\nabla W(t,x)\cdot\nabla W_t(t,x)\mathrm{d}x\\
&+\tau k_2\int_{\mb{R}^n}W_{ttt}(t,x)W_t(t,x)\mathrm{d}x+(1+\gamma\tau)k_2\int_{\mb{R}^n}W_{tt}(t,x)W_t(t,x)\mathrm{d}x\\
&\geqslant -\frac{(\gamma+1)^2}{4\varepsilon_4}\|\nabla W_t(t,\cdot)\|_{L^2(\mb{R}^n)}^2-\varepsilon_4\|\nabla W_{tt}(t,\cdot)\|_{L^2(\mb{R}^n)}^2-\frac{(1-\gamma)^2}{4\varepsilon_5}\|\nabla W(t,\cdot)\|_{L^2(\mb{R}^n)}^2\\
&\quad-\varepsilon_5\|\nabla W_{tt}(t,\cdot)\|_{L^2(\mb{R}^n)}^2-\frac{\tau^2k_1^2}{4\varepsilon_6}\|W_{ttt}(t,\cdot)\|_{L^2(\mb{R}^n)}^2-\varepsilon_6\|W_{tt}(t,\cdot)\|_{L^2(\mb{R}^n)}^2\\
&\quad-\frac{(1-\gamma)^2k_1^2}{4\varepsilon_7}\|\nabla W(t,\cdot)\|_{L^2(\mb{R}^n)}^2-\varepsilon_7\|\nabla W_{t}(t,\cdot)\|_{L^2(\mb{R}^n)}^2-\frac{\tau^2k_2^2}{4\varepsilon_8}\|W_{ttt}(t,\cdot)\|_{L^2(\mb{R}^n)}^2\\
&\quad-\varepsilon_8\|W_t(t,\cdot)\|_{L^2(\mb{R}^n)}^2-\frac{(1+\gamma\tau)^2k_2^2}{4\varepsilon_9}\|W_{tt}(t,\cdot)\|_{L^2(\mb{R}^n)}^2-\varepsilon_9\|W_t(t,\cdot)\|_{L^2(\mb{R}^n)}^2.
\end{align*}
Thus, we can estimate
\begin{align*}
\widetilde{\ml{E}}_m[W](t)&\geqslant\tau a_{m1}\|W_{ttt}(t,\cdot)\|_{L^2(\mb{R}^n)}^2+a_{m2}\|\nabla W_{tt}(t,\cdot)\|_{L^2(\mb{R}^n)}^2+a_{m3}\|W_{tt}(t,\cdot)\|_{L^2(\mb{R}^n)}^2\\
&\quad+a_{m4}\|\nabla W_{t}(t,\cdot)\|_{L^2(\mb{R}^n)}^2+a_{m5}\| W_{t}(t,\cdot)\|_{L^2(\mb{R}^n)}^2+a_{m6}\|\nabla W(t,\cdot)\|_{L^2(\mb{R}^n)}^2,
\end{align*}
where the coefficients are given by
\begin{align*}
a_{m1}:&=\frac{1}{2}-\frac{\tau k_1^2}{4\varepsilon_6}-\frac{\tau k_2^2}{4\varepsilon_8},\ \ a_{m2}:=\frac{1}{2}-\varepsilon_4-\varepsilon_5,\\
a_{m3}:&=\frac{\gamma}{2}+\frac{(1+\gamma\tau)k_1}{2}-\frac{\tau k_2}{2}-\varepsilon_3-\varepsilon_6-\frac{(1+\gamma\tau)^2k_2^2}{4\varepsilon_9},\\
a_{m4}:&=\frac{1-\gamma}{2}+\frac{(\gamma+1)k_1}{2}+\frac{k_2}{2}-\frac{(\gamma+1)^2}{4\varepsilon_4}-\varepsilon_7,\ \ a_{m5}:=\frac{\gamma k_2}{2}-\varepsilon_3-\varepsilon_8-\varepsilon_9\\
a_{m6}:&=\frac{(\gamma-1)k_2}{2}-\varepsilon_3-\frac{(1-\gamma)^2}{4\varepsilon_5}-\frac{(1-\gamma)^2k_1^2}{4\varepsilon_7}.
\end{align*}
As what we did in \eqref{Parameter_Range_01}, we choose $\varepsilon_3,\varepsilon_6,\varepsilon_8$ to be sufficient small and fixed positive constants, $0<\tau\ll 1$, $k_1=\gamma+1$ and $k_2\uparrow\gamma+\gamma^2$. Then, in the case $\gamma>5$, all coefficients $a_{mj}$ for $j=1,\dots,6$, are positive number provided $\varepsilon_4\downarrow 1/3$, $\varepsilon_5\uparrow 1/6$, $\varepsilon_7\uparrow(\gamma-1)^2/4$, $\varepsilon_9\uparrow2\gamma(\gamma^2+\gamma)/5$ and $\varepsilon_4+\varepsilon_5\leqslant 1/2$. We should underline that the suitable choice of $\varepsilon_4$ is determined by the positive condition from $a_{m4}$. What's more, the choice of $k_1$ and $k_2$ are restricted by the condition \eqref{Parameter_Range_01}.\\
 All in all, for $\gamma>5$, there exists a positive constant $C$ independent of $\tau$ such that
\begin{align*}
\widetilde{\ml{E}}_m[W](t)\geqslant C\|W_t(t,\cdot)\|_{L^2(\mb{R}^n)}^2.
\end{align*}
By applying \eqref{Rela_Sing_Lim_Solu}, our proof is complete. 
\end{proof}

\section*{Acknowledgments}
The author thanks Yanan Li (Harbin Engineering University)  for the suggestions in the preparation of the paper.

\end{document}